\documentclass[a4paper]{article} \usepackage{etoolbox} \newtoggle{author} \toggletrue{author} \newtoggle{arxiv} \toggletrue{arxiv}

\iftoggle{author}{
    \usepackage[utf8]{inputenc}
    \usepackage[T1]{fontenc}
    \usepackage[english]{babel}

    \usepackage{amsmath}
    \usepackage{amssymb}
    \usepackage{amsthm}
    \usepackage{enumitem}
    \usepackage{geometry}
    \usepackage{graphicx}
    \usepackage{hyperref}
    \usepackage{lineno}
    \usepackage{microtype}
    \usepackage{tikz}

    \setlist[itemize]{leftmargin=*}
}{}

\usepackage{placeins}

\iftoggle{author}{\theoremstyle{definition}}{\theoremstyle{thmstyletwo}}
\newtheorem{definition}{Definition}[section]
\newtheorem{assumption}[definition]{Assumption}

\iftoggle{author}{\theoremstyle{plain}}{\theoremstyle{thmstyletwo}}

\newtheorem{lemma}[definition]{Lemma}
\newtheorem{proposition}[definition]{Proposition}
\newtheorem{theorem}[definition]{Theorem}

\iftoggle{author}{\theoremstyle{remark}}{\theoremstyle{thmstyletwo}}
\newtheorem{remark}[definition]{Remark}

\numberwithin{equation}{section}

\iftoggle{author}{
    \iftoggle{arxiv}{}{
        \cspreto{equation*}{\linenomath}
        \cspreto{endequation*}{\endlinenomath}
        \cspreto{align}{\linenomath}
        \cspreto{endalign}{\endlinenomath}
        \cspreto{align*}{\linenomath}
        \cspreto{endalign*}{\endlinenomath}

        \linenumbers
    }
}{}

\newcommand{\MM}{\mathbb{M}}
\newcommand{\NN}{\mathbb{N}}
\newcommand{\PP}{\mathbb{P}}
\newcommand{\RR}{\mathbb{R}}

\newcommand{\bn}{\mathbf{n}}

\newcommand{\cA}{\mathcal{A}}

\newcommand{\cK}{\mathcal{K}}
\newcommand{\cL}{\mathcal{L}}

\newcommand{\cS}{\mathcal{S}}

\newcommand{\<}{\langle}
\renewcommand{\>}{\rangle}

\let\epsilon\undefined
\let\phi\undefined

\DeclareMathOperator{\SO}{SO}

\DeclareMathOperator{\diam}{diam}
\DeclareMathOperator{\dofs}{dofs}

\DeclareMathOperator{\intt}{int}

\DeclareMathOperator{\tr}{tr}

\begin{document}

\iftoggle{author}{}{
    \DOI{DOI HERE}
    \copyrightyear{2026}
    \vol{00}
    \pubyear{2026}
    \access{Advance Access Publication Date: Day Month Year}
    \appnotes{Paper}
    \copyrightstatement{Published by Oxford University Press on behalf of the Institute of Mathematics and its Applications. All rights reserved.}
    \firstpage{1}

}

\iftoggle{author}{
    \title{Virtual element methods for a class of fully nonlinear elliptic PDEs}
}{
    \title[Virtual element methods for a class of fully nonlinear elliptic PDEs]{Virtual element methods for a class of fully nonlinear elliptic PDEs}
}

\iftoggle{author}{
    \author{
        Guillaume Bonnet\footnote{CEREMADE, CNRS, Université Paris--Dauphine, Université PSL, 75016 Paris, France}
        \and
        Andrea Cangiani\footnote{Mathematics Area, International School for Advanced Studies (SISSA), 34136 Trieste, Italy}
        \and
        Andreas Dedner\footnote{Mathematics Institute, University of Warwick, Coventry, CV4 7AL, UK}
        \and
        Ricardo H.~Nochetto\footnote{Department of Mathematics and Institute for Physical Science and Technology, University of Maryland, College Park, MD 20742, USA}
    }
}{
    \author{Guillaume Bonnet
    \address{\orgdiv{CEREMADE}, \orgname{CNRS}, \orgname{Université Paris--Dauphine}, \orgname{Université PSL}, \orgaddress{\postcode{75016} Paris, \country{France}}}}
    \author{Andrea Cangiani*
    \address{\orgdiv{Mathematics Area}, \orgname{International School for Advanced Studies (SISSA)}, \orgaddress{\postcode{34136} Trieste, \country{Italy}}}}
    \author{Andreas Dedner
    \address{\orgdiv{Mathematics Institute}, \orgname{University of Warwick}, \orgaddress{Coventry, \postcode{CV4 7AL}, \country{UK}}}}
    \author{Ricardo H.~Nochetto
    \address{\orgdiv{Department of Mathematics} and \orgdiv{Institute for Physical Science and Technology}, \orgname{University of Maryland}, \orgaddress{College Park, \state{MD} \postcode{20742}, \country{USA}}}}

    \authormark{Guillaume Bonnet, Andrea Cangiani, Andreas Dedner, and Ricardo H.~Nochetto}

    \corresp[*]{Corresponding author: \href{email:acangian@sissa.it}{acangian@sissa.it}}
}

\iftoggle{arxiv}{\date{}}

\iftoggle{author}{}{
    \received{Date}{0}{Year}
    \revised{Date}{0}{Year}
    \accepted{Date}{0}{Year}

}

\iftoggle{author}{\maketitle}{}

\iftoggle{author}{\begin{abstract}}{\abstract}{We study virtual element discretizations of a well-known variational formulation in $H^2$ of Hamilton--Jacobi--Bellman and Isaacs equations with Cordes coefficients. We show that the use of $H^2$ conforming virtual element spaces leads to a relatively simple analysis, bypassing the need for discrete Miranda--Talenti estimates that typically arises when using nonconforming schemes.
We investigate how the polynomial degree of projection operators, especially for lower order terms, affects both the error analysis and robustness of the proposed schemes.
We also show the possibility of imposing weakly the Dirichlet boundary condition, which simplifies the implementation of the method in some virtual element codes. Our results are complemented by numerical experiments in which we compare the convergence of different variants of the scheme for some test problems.}\iftoggle{author}{\end{abstract}}{}

\iftoggle{author}{}{


    \maketitle
}


\section{Introduction}

In this work we are concerned with the efficient discretization of  a general class of Isaacs fully nonlinear equations, which encompass the Hamilton--Jacobi--Bellman (HJB) equations as a special case. Fully nonlinear elliptic partial differential equations play a central role in several areas of applied mathematics, including optimal control, differential geometry, differential games, and image processing~\cite{falcone2013,maugeri2000}. The numerical approximation of such problems remains challenging due to their nondivergence structure, strong nonlinearity, and the limited regularity typically available for their solutions~\cite{motzkin1953,debrabant2013,falcone2013,smears2013,smears2014,gallistl2017,neilan2019,wu2021,kawecki2021,gallistl2025,bonnans2004,cai2025}.

We propose and analyze arbitrary-order $H^2$-conforming Virtual Element Methods (VEM) for uniformly elliptic Isaacs equations.  The approach builds upon our previous work on $H^2$-conforming VEM for linear elliptic problems in nondivergence form~\cite{bonnet2025}, thereby extending the methodology to the fully nonlinear setting, and accommodating for the presence of lower order terms.

More specifically,  we  seek approximations of strong solutions of Isaacs equations satisfying the Cordes condition on convex domains. In this setting, well-posedness is a consequence of the Miranda--Talenti inequality, and is established through a variational formulation in $H^2$ (rather than $H^1$). This has been used in the literature to design discontinuous Galerkin~\cite{smears2013,smears2014} and other, mostly nonconforming finite element methods~\cite{gallistl2017,neilan2019,kawecki2021,wu2021,dedner2022nonvar,gallistl2025,cai2025}.

The Cordes condition amounts to uniform ellipticity in dimension two, and is more restrictive in higher dimensions. For completeness, it may be observed that another family of numerical schemes for fully nonlinear elliptic equations are monotone schemes, see e.g.~\cite{bonnans2004,debrabant2013,falcone2013,bonnans2023}. Monotone schemes do not require the Cordes condition or even uniform ellipticity, and approximate viscosity solutions rather than strong solutions; however, their order is limited to two~\cite{motzkin1953}. Accordingly, the two families of methods tend to be favored in different settings.

As already observed in~\cite{gallistl2017,bonnet2025}, in the Cordes setting, the use of $H^2$-conforming discrete spaces enables a straightforward discretization of second-order operators in strong form, since the continuous Miranda--Talenti inequality directly applies at the discrete level.
By contrast, the analysis of nonconforming finite element discretizations requires a suitable discrete analogue of Miranda--Talenti which is typically achieved by the inclusion of appropriate stabilizing terms.

The following topics are the main contributions of this paper.
\begin{itemize}
\item {\bf VEMs for Isaacs equations}.
We propose a family of VEM based on $H^2$-conforming virtual element spaces~\cite{brezzi2013,beirao2014,chen2022conforming}
Our schemes are based on equal-order polynomial projections so as to naturally maintain the structural balance required for stability.
Specifically, for any given target polynomial order $m\ge 2$, we consider the VEM scheme obtained by applying elementwise $L^2$-projections of order $m-j$ with $j\in\{0,1,2\}$ to \emph{all} terms.
%
\item {\bf Monotonicity and error estimates.} We establish, under minimal regularity assumptions, strong monotonicity of the respective discrete nonlinear forms. These monotonicity results hinge on the Cordes condition and hold provided that the classical VEM stabilization term is chosen sufficiently large. The monotonicity is the basis for providing well-posedness and deriving optimal order-regularity error estimates in an $H^2$-type norm.
Our analysis shows that the choice of polynomial order for the underlying VEM projection operators is more delicate for the nonlinear setting than for the linear case~\cite{bonnet2025}. In fact, not all combinations of projection orders are admissible, or  favorable, for the diffusion, advection and reaction terms appearing in the equation.
To illustrate this point, we also investigate a VEM employing non-equal projection orders and show that, in this case, strong monotonicity depends more strongly on the Cordes condition.

\item {\bf Weak imposition of boundary conditions.} We propose and analyze a Nitsche-type boundary penalty method to impose Dirichlet boundary conditions weakly. Since the variational formulation does not rely on integration by parts, weak imposition of the boundary condition is particularly natural and leads to simpler implementation while preserving  the convergence order.  However, the analysis of the resulting scheme is nontrivial and relies on a new variant of the Miranda--Talenti inequality involving a nonstandard trace operator for  $H^2$ functions with nonzero trace~\cite{bernardi2007}.

\item {\bf Numerical experiments.} We complement theory with exhaustive numerical experiments that consider all variants of the VEMs and examine the polynomial order of projection operators as well as strong and weak imposition of boundary conditions. The experiments are performed using the unified implementation provided by the Dune-VEM code~\cite{dedner2022fourthorder,dedner2024}, exploiting the constrained least squares approach providing flexible polynomial projections of any order.
We compare the proposed VEM schemes among themselves and against the discontinuous Galerkin method proposed in~\cite{smears2014}. Our results indicate that  the variant employing the higher-order projections has improved accuracy in the presence of a dominating convective term. The methods with weakly imposed boundary conditions perform comparably to their strongly imposed counterparts, confirming that weak imposition provides a computationally simpler yet equally effective alternative.
We also include a study of the number of iterations required by the Howard algorithm, which is used to solve the discrete nonlinear problems. Our results show that convergence is superlinear in all cases, with comparable number of iterations (typically, 2 to 6 iterations to achieve a tolerance of $10^{-9}$) for the VEM and DG methods.
\end{itemize}

All the $H^2$-conforming VEM variants considered in this work are based directly on the variational formulation obtained by testing the strong form of the PDE with a suitable linear operator, as originally proposed in~\cite{smears2014}. Consequently, no additional  stabilizing terms are introduced beyond the standard VEM stabilization term already required in the linear case. However, the stabilization parameter must be chosen sufficiently large to guarantee monotonicity. A different VEM has been proposed recently in~\cite{cai2025} for HJB equations in two space dimensions. In that work, both conforming and nonconforming methods are constructed 
by augmenting the basic formulation with ad hoc positive forms, so as to enforce monotonicity independently of the Cordes conditions. 
The methods in~\cite{cai2025} employ equal-order projections of order  $m-2$, similarly to one of the variants studied here. The numerical comparisons reported in the final section may also provide useful indications for the behavior of the methods proposed in~\cite{cai2025}, particularly in the convection-dominated case.

Another attractive feature of $H^2$-conforming VEM is the direct availability of the gradient of the discrete solution. This may be advantageous both for the nonlinear solution procedure and for driving adaptive mesh refinement, which is itself facilitated by the ability of the VEM to accommodate general polygonal meshes.

This paper is organized as follows. In section~\ref{sec:model} we introduce the fully nonlinear Isaacs equation along with the Cordes condition on its coefficients. We show a crucial consequence that leads to well-posedness of the variational formulation that arises from testing the equation with a suitable second order linear operator. In section~\ref{sec:vem_framework} we describe the $H^2$-conforming virtual element framework, whereas in section~\ref{sec:scheme} we present three VEMs based on equal-order piece-wise projection operators $\Pi_h^{m-j}$, where $j = 0, 1, 2$ and $m \ge 2$ is the polynomial degree. We analyze the VEMs in section~\ref{sec:analysis}: we define a suitable $H^2$-discrete norm, prove strong monotonicity, existence and uniqueness, and establish an abstract error estimate. We present two variants of the basic schemes in section~\ref{sec:variants}. We first examine the effect of using projection operators with different polynomial degrees, and next study the weak imposition of Dirichlet boundary conditions via a penalty formulation. After this somewhat abstract presentation, we devote section~\ref{sec:realvem} to discussing a specific VEM concrete realization and deriving optimal order-regularity error estimates. We conclude in section~\ref{sec:numerics} with several numerical experiments that explore properties of the proposed VEMs, and compare them among themselves and against the discontinuous Galerkin method of~\cite{smears2014}.

\section{The continuous problem}\label{sec:model}

Let us denote by $\cS_d$ (respectively $\cS_d^{++}$) the set of symmetric (respectively symmetric positive definite) matrices of size $d$. We equip the space $\cS_d$ with the Frobenius norm and inner product.

We look for a strong solution $u \in H^2(\Omega)$ to the Isaacs equation
\begin{equation}
    \label{eq:isaacs}
    \inf_{\alpha \in \Lambda_1} \sup_{\beta \in \Lambda_2}\, [A^{\alpha, \beta}(x) : \nabla^2 u(x) + b^{\alpha, \beta}(x) \cdot \nabla u(x) - c^{\alpha, \beta}(x) u(x) - f^{\alpha, \beta}(x)] = 0 \quad \text{a.e.\ in } \Omega,
\end{equation}
where $\Omega \subset \RR^d$ is a bounded convex polytope, $\Lambda_1$ and $\Lambda_2$ are compact metric spaces, and $A \colon \Lambda_1 \times \Lambda_2 \times \Omega \to \cS_d^{++}$, $b \colon \Lambda_1 \times \Lambda_2 \times \Omega \to \RR^d$, $c \colon \Lambda_1 \times \Lambda_2 \times \Omega \to \RR_+$, and $f \colon \Lambda_1 \times \Lambda_2 \times \Omega \to \RR$ are functions satisfying the Carathéodory conditions, i.e.\ they are measurable with respect to $x \in \Omega$ for every $(\alpha, \beta) \in \Lambda_1 \times \Lambda_2$ and continuous with respect to $(\alpha, \beta) \in \Lambda_1 \times \Lambda_2$ for almost every $x \in \Omega$. We impose the Dirichlet boundary condition
\begin{equation}
    \label{eq:bc}
    u = g \quad \text{a.e.\ on } \partial \Omega,
\end{equation}
for some given $g \in H^2(\Omega)$.

We assume that $A$, $b$, and $c$ satisfy the \emph{Cordes condition}:

\begin{assumption}[Cordes condition]
    \label{assum:cordes}
    Either one of the two following properties holds:
    \begin{itemize}
        \item There exist $\lambda > 0$ and $\varepsilon \in (0, 1]$ such that for any $\alpha \in \Lambda_1$, $\beta \in \Lambda_2$, and $x \in \Omega$, one has
        \begin{equation}
            \label{eq:cordes_low_order_case}
            \frac{|A^{\alpha, \beta}(x)|^2 + |b^{\alpha, \beta}(x)|^2 / (2 \lambda) + (c^{\alpha, \beta}(x) / \lambda)^2}{(\tr A^{\alpha, \beta}(x) + c^{\alpha, \beta}(x) / \lambda)^2} \leq \frac{1}{d + \varepsilon}.
        \end{equation}
        \item There exists $\varepsilon \in (0, 1]$ such that for any $\alpha \in \Lambda_1$, $\beta \in \Lambda_2$, and $x \in \Omega$, one has
        \begin{align}
            \label{eq:cordes_high_order_case}
            \frac{|A^{\alpha, \beta}(x)|^2}{(\tr A^{\alpha, \beta}(x))^2} &\leq \frac{1}{d - 1 + \varepsilon}, &
            b^{\alpha, \beta}(x) &= 0, &
            c^{\alpha, \beta}(x) &= 0.
        \end{align}
        In this case, we let $\lambda := 0$.
    \end{itemize}
\end{assumption}

We will not use the Cordes condition directly, but rather its following consequence.

\begin{proposition}
    \label{prop:cordes_charac}
    There exist $\lambda \geq 0$, $\kappa \in (0, 1]$, and a function $\gamma \colon \Lambda_1 \times \Lambda_2 \times \Omega \to \RR_+^*$, measurable with respect to $x \in \Omega$ for every $(\alpha, \beta) \in \Lambda_1 \times \Lambda_2$ and continuous with respect to $(\alpha, \beta) \in \Lambda_1 \times \Lambda_2$ for almost every $x \in \Omega$, such that, for any $\alpha \in \Lambda_1$, $\beta \in \Lambda_2$, and $x \in \Omega$,
    \begin{equation}
        \label{eq:cordes_ineq_low_order_case}
        \lambda^2 |\gamma^{\alpha, \beta}(x) A^{\alpha, \beta}(x) - I_d|^2 + \lambda |\gamma^{\alpha, \beta}(x) b^{\alpha, \beta}(x)|^2 / 2 + |\gamma^{\alpha, \beta}(x) c^{\alpha, \beta}(x) - \lambda|^2 \leq \lambda^2 (1 - \kappa)^2,
    \end{equation}
    and in the case $\lambda = 0$,
    \begin{align}
        \label{eq:cordes_ineq_high_order_case}
        |\gamma^{\alpha, \beta}(x) A^{\alpha, \beta}(x) - I_d|^2 &\leq (1 - \kappa)^2, &
        b^{\alpha, \beta}(x) &= 0, &
        c^{\alpha, \beta}(x) &= 0.
    \end{align}
\end{proposition}

\begin{proof}
    It is easily verified that when letting $\kappa := 1 - \sqrt{1 - \varepsilon}$ and
    \begin{equation}
        \label{eq:gamma}
        \gamma^{\alpha, \beta}(x) := \begin{cases}
            \left(\tr A^{\alpha, \beta}(x) + \frac{c^{\alpha, \beta}(x)}{\lambda}\right) / \left(|A^{\alpha, \beta}(x)|^2 + \frac{|b^{\alpha, \beta}(x)|^2}{2 \lambda} + \frac{c^{\alpha, \beta}(x)^2}{\lambda^2}\right) &\text{if } \lambda > 0, \\
            \tr A^{\alpha, \beta}(x) / |A^{\alpha, \beta}(x)|^2 &\text{if } \lambda = 0,
        \end{cases}
    \end{equation}
    the inequalities \eqref{eq:cordes_ineq_low_order_case} and \eqref{eq:cordes_ineq_high_order_case} reduce respectively to \eqref{eq:cordes_low_order_case} and \eqref{eq:cordes_high_order_case}.
\end{proof}

\begin{remark}
    For any $\alpha \in \Lambda_1$, $\beta \in \Lambda_2$, and $x \in \Omega$, the quantity $\gamma^{\alpha, \beta}(x)$ defined in \eqref{eq:gamma} is minimizing the left-hand side of \eqref{eq:cordes_ineq_high_order_case} if $\lambda > 0$, and is minimizing the left-hand side of the first inequality in \eqref{eq:cordes_ineq_low_order_case} if $\lambda = 0$, as can be verified by writing the relevant first-order optimality condition. Therefore, the statement of Proposition~\ref{prop:cordes_charac} is actually a characterization of the Cordes condition. The presence of the additional conditions \eqref{eq:cordes_ineq_high_order_case} in the case $\lambda = 0$ can be interpreted as follows: if $\lambda > 0$, then \eqref{eq:cordes_ineq_low_order_case} implies that $|\gamma^{\alpha, \beta}(x) A^{\alpha, \beta}(x) - I_d|^2 \leq (1 - \kappa)^2$, $|\gamma^{\alpha, \beta}(x) b^{\alpha, \beta}(x)|^2 / 2 \leq \lambda (1 - \kappa)^2$, and $|\gamma^{\alpha, \beta}(x) c^{\alpha, \beta}(x) - \lambda|^2 \leq \lambda^2 (1 - \kappa)^2$; \eqref{eq:cordes_ineq_high_order_case} results from letting $\lambda$ approach zero in these inequalities.
\end{remark}

From now on, we let $\lambda \geq 0$, $\kappa \in (0, 1]$, and $\gamma \colon \Lambda_1 \times \Lambda_2 \times \Omega \to \RR$ be as in Proposition~\ref{prop:cordes_charac}. In practice, the function $\gamma$ may be chosen according to \eqref{eq:gamma}, although we also allow other choices of functions $\gamma$ as long as the inequalities \eqref{eq:cordes_ineq_low_order_case} and \eqref{eq:cordes_ineq_high_order_case} are satisfied.

One consequence of \eqref{eq:cordes_ineq_low_order_case} and \eqref{eq:cordes_ineq_high_order_case} is that the functions
\begin{align*}
    x &\mapsto \sup_{\substack{\alpha \in \Lambda_1 \\ \beta \in \Lambda_2}}\, \gamma^{\alpha, \beta}(x) |A^{\alpha, \beta}(x)|, &
    x &\mapsto \sup_{\substack{\alpha \in \Lambda_1 \\ \beta \in \Lambda_2}}\, \gamma^{\alpha, \beta}(x) |b^{\alpha, \beta}(x)|, &
    x &\mapsto \sup_{\substack{\alpha \in \Lambda_1 \\ \beta \in \Lambda_2}}\, \gamma^{\alpha, \beta}(x) c^{\alpha, \beta}(x),
\end{align*}
all belong to $L^\infty(\Omega)$. We need to assume a similar condition on the source term $f$:

\begin{assumption}[source term]
    The function $x \mapsto \sup_{\alpha \in \Lambda_1,\, \beta \in \Lambda_2}\, \gamma^{\alpha, \beta}(x) |f^{\alpha, \beta}(x)|$ belongs to $L^2(\Omega)$.
\end{assumption}

In this setting, the analysis of the Dirichlet problem can be performed by relying on an equivalent variational formulation. In order to describe it, let us first define, for any $v \in H^2(\Omega)$, the pointwise operators
\begin{align}
    \label{eq:pointwise_operators}
    \begin{split}
        \cL^{\alpha, \beta} v &:= A^{\alpha, \beta} : \nabla^2 v + b^{\alpha, \beta} \cdot \nabla v - c^{\alpha, \beta} v, \\
        F[v] &:= \inf_{\alpha \in \Lambda_1} \sup_{\beta \in \Lambda_2}\, [\cL^{\alpha, \beta} v - f^{\alpha, \beta}],
    \end{split} &
    \begin{split}
        \cL_\lambda v &:= \Delta v - \lambda v, \\
        F_\gamma[v] &:= \inf_{\alpha \in \Lambda_1} \sup_{\beta \in \Lambda_2}\, [\gamma^{\alpha, \beta} (\cL^{\alpha, \beta} v - f^{\alpha, \beta})]
    \end{split}
\end{align}
(in our analysis, we will also apply these operators to functions belonging to some broken $H^2$ space; we naturally extend their definitions to this case). Then the Isaacs equation \eqref{eq:isaacs} can be written as $F[u] = 0$ a.e.\ in $\Omega$, and is equivalent to $F_\gamma[u] = 0$ a.e.\ in $\Omega$. Moreover, it is easily verified that $F_\gamma[v] \in L^2(\Omega)$ for any $v \in H^2(\Omega)$.

Let us then introduce the operator $\cA \colon H^2(\Omega) \to (H^2(\Omega))^*$ defined by
\begin{equation}
    \label{eq:cA}
    \<\cA(v), w\> := \int_\Omega F_\gamma[v] \cL_\lambda w\, d x,
\end{equation}
and the spaces
\begin{align*}
    V^g &:= \{v \in H^2(\Omega) \mid v_{|\partial \Omega} = g_{|\partial \Omega}\}, &
    V^0 &:= H^2(\Omega) \cap H_0^1(\Omega) = \{v \in H^2(\Omega) \mid v_{|\partial \Omega} = 0\}.
\end{align*}
Then the problem \eqref{eq:isaacs}, with the boundary condition \eqref{eq:bc}, is equivalent to the variational problem
\begin{equation}
    \label{eq:var}
    \text{find } u \in V^g \text{ s.t.\ } \forall v \in V^0,\, \<\cA(u), v\> = 0.
\end{equation}
The equivalence follows from the fact that $\{\cL_\lambda v \mid v \in V^0\} = L^2(\Omega)$, so testing against $\cL_\lambda v$ for any $v \in V^0$ amounts to testing against any $L^2$ function.

Under our assumptions, it is well-known that there exists a unique solution $u \in V^g$ to \eqref{eq:var}. One may refer to \cite{kawecki2021} for a proof, 
or, equivalently, follow the sketch of the proof that we use in section~\ref{sec:analysis} for the well-posedness of our discrete variational problem \eqref{eq:scheme}.

\section{Virtual element framework}
\label{sec:vem_framework}

In this section, we describe the $H^2$ conforming virtual element framework that we will use in section~\ref{sec:scheme} as a basis for our discretization of the variational problem \eqref{eq:var}.

We assume that we are given a polytopal mesh $\cK_h$ discretizing $\Omega$, and an arbitrary polynomial consistency order $m \geq 2$. We associate some characteristic length $h_F > 0$ to each $d'$-dimensional face $F$ of each cell $K \in \cK_h$, for $0 \leq d' \leq d$, thus including $K$ itself, and assume that the following standard shape-regularity assumption holds.

\begin{assumption}[shape-regularity]\label{ass:reg}
    There exists $\rho \in (0, 1)$ such that any $d'$-dimensional face $F$ of any $K \in \cK_h$, with $0 \leq d' \leq d$, is star-shaped with respect to a $d'$-dimensional ball of diameter $\rho \diam(K)$, and its characteristic length satisfies $\rho \diam(K) \leq h_F \leq \rho^{-1} \diam(K)$.
\end{assumption}

From now on, for any two quantities $a_1$, $a_2 \in \RR$, we write $a_1 \lesssim a_2$ (respectively $a_1 \gtrsim a_2$) if there exists some constant $C > 0$, depending only on $d$, $\Omega$, the polynomial consistency order $m$, and the shape-regularity parameter $\rho$, such that $a_1 \leq C a_2$ (respectively $a_1 \geq C a_2$). We write $a_1 \simeq a_2$ if both $a_1 \lesssim a_2$ and $a_1 \gtrsim a_2$.

We let $\Omega_h := \bigcup_{K \in \cK_h} \intt(K)$, so as to be able to work with the broken $H^2$ space $H^2(\Omega_h)$. Observe that $H^2(\Omega) \subset H^2(\Omega_h) \subset L^2(\Omega)$.

For any polytope $\omega$ (be it of dimension $d$ or less) and any $k \in \NN$, we denote by $\Pi_\omega^k$ the $L^2$ projection operator onto the space $\PP_k(\omega)$ of polynomials of degree $k$ on $\omega$. For any $k \in \NN$, we define the space
\begin{equation*}
    \PP_{k, h}(\Omega) := \{p \in L^2(\Omega) \mid \forall K \in \cK_h,\, p_{|K} \in \PP_k(K)\}
\end{equation*}
of discontinuous piecewise polynomial functions of degree $k$ on the mesh $\cK_h$, and denote by $\Pi_h^k$ the $L^2$ projection operator onto $\PP_{k, h}(\Omega)$. We naturally extend the definitions of the projection operators $\Pi_\omega^k$ and $\Pi_h^k$ to the cases of vector- and matrix-valued functions.

We assume that we are given:
\begin{itemize}
    \item A finite-dimensional linear subspace $V_h \subset H^2(\Omega)$; we then let $V_h^0 := V_h \cap H_0^1(\Omega)$.
    \item A function $g_I \in V_h$, to be interpreted as some interpolant of the function $g$ from the Dirichlet boundary condition \eqref{eq:bc}; we then let $V_h^g := g_I + V_h^0$.
    \item Symmetric bilinear forms $s_{0, h} \colon (V_h + \PP_{m, h}(\Omega))^2 \to \RR$ and $s_{2, h} \colon (V_h + \PP_{m, h}(\Omega))^2 \to \RR$.
\end{itemize}

In order to accommodate for the several variants of the $H^2$ conforming virtual element method described in the literature and implemented in computer codes, we do not prescribe a specific construction of $V_h$, $g_I$, $s_{0, h}$, and $s_{2, h}$, but only impose that the following assumptions hold. However, later on in section~\ref{sec:realvem} we shall provide the concrete realization of the $H^2$ conforming framework which is tested with a series of numerical examples in section~\ref{sec:numerics}.

\begin{assumption}[inverse inequalities]
    \label{assum:inverse_ineq}
    For any $v \in V_h$, cell $K \in \cK_h$, and side $s$ of $K$,
    \begin{align}
        \label{eq:inverse_ineq_cell}
        |v|_{j, K} &\lesssim h_K^{i-j} |v|_{i, K}, \\
        \label{eq:inverse_ineq_side}
        |v|_{j, s} &\lesssim h_s^{i-j} \|v\|_{i, s},
    \end{align}
for $0\le i<j\le 2$.
\end{assumption}

\begin{assumption}[stabilization form scaling]
    \label{assum:stab_scaling}
    For any $v \in V_h + \PP_{m, h}(\Omega)$, one has
    \begin{equation}
        \label{eq:stab_scaling0}
        s_{0, h}(v, v) \eqsim \|v\|_0^2.
    \end{equation}
    If $\Pi_h^1 v = 0$, then one also has
    \begin{equation}
        \label{eq:stab_scaling2}
        s_{2, h}(v, v) \eqsim |v|_2^2.
    \end{equation}
\end{assumption}

\begin{remark}[requirement for inverse inequalities]
    We will only need the inverse inequalities \eqref{eq:inverse_ineq_cell} and \eqref{eq:inverse_ineq_side} in section~\ref{subsec:nitsche}, and not in the analysis of our main numerical scheme \eqref{eq:scheme}.
\end{remark}

One important feature of the virtual element method is that functions of $V_h$ are typically \emph{virtual}, in the sense that their pointwise values cannot be efficiently computed from the values of the degrees of freedom that characterize them. Rather, we only assume the efficient computability of the discontinuous piecewise polynomial functions $\Pi_h^{m-j} v$, $\Pi_h^{m-j} \nabla v$, and $\Pi_h^{m-j} \nabla^2 v$, for $v \in V_h$ and some $j\in\{0,1,2\}$, and of $s_{0, h}(v, w)$ and $s_{2, h}(v, w)$, for $v$, $w \in V_h + \PP_{m, h}(\Omega)$.

\begin{remark}[degrees of the projections]
    Choosing the degree $m$ for all projections is nonstandard: usually, for $H^2$ conforming virtual element functions, only the efficient computability of $\Pi_h^{m-1} \nabla v$ and $\Pi_h^{m-2} \nabla^2 v$ is assumed.
However, as show below in section~\ref{sec:numerics}, all projections mentioned above can be computed exploiting the constrained least squares approach from~\cite{dedner2022fourthorder,dedner2024}.
\end{remark}

\section{The virtual element scheme}
\label{sec:scheme}

Since the quantity $\<\cA(v), w\>$ defined in \eqref{eq:cA} is typically not efficiently computable from the values of the degrees of freedom associated to functions $v \in V_h^g$ and $w \in V_h^0$, we need to design a variant of the operator $\cA$ suitable for use in our virtual element scheme based on computable polynomial projections. We actually propose a family of methods depending on the order of such projections.

First,  we define, for $v \in H^2(\Omega_h)$, the following counterparts to the operators $\cL^{\alpha, \beta}$, $\cL_\lambda$, and $F_\gamma$ from \eqref{eq:pointwise_operators}:
\begin{equation}
    \label{eq:projected_pointwise_operators}
    \begin{split}
        \cL_h^{\alpha, \beta} v &:= A^{\alpha, \beta} : \Pi_h^{m-j} \nabla^2 v + b^{\alpha, \beta} \cdot \Pi_h^{m-j} \nabla v - c^{\alpha, \beta} \Pi_h^{m-j} v, \\
        \cL_{\lambda, h} v &:= \Pi_h^{m-j} (\Delta v - \lambda v), \\
        F_{\gamma, h}[v] &:= \inf_{\alpha \in \Lambda_1} \sup_{\beta \in \Lambda_2}\, [\gamma^{\alpha, \beta} (\cL_h^{\alpha, \beta} v - f^{\alpha, \beta})],
    \end{split}
\end{equation}
for some given $j\in\{0,1,2\}$, and where $\nabla$, $\nabla^2$, and $\Delta$ denote respectively the broken gradient, Hessian, and Laplacian. Here $j\in\{0,1,2\}$ is a parameter of the method that can be chosen freely, and will slightly affect our error estimate (see Theorem~\ref{thm:error_estimate}).

We then define operators $\cA_h^c$, $\cA_h^s$, and $\cA_h$ from $V_h + \PP_{m, h}(\Omega)$ to $(V_h)^*$ by letting, for $v \in V_h + \PP_{m, h}(\Omega)$ and $w \in V_h$,
\begin{align*}
    \<\cA_h^c(v), w\> &:= \int_\Omega F_{\gamma, h}[v] \cL_{\lambda, h} w\, d x, \\
    \<\cA_h^s(v), w\> &:= s_{2, h}(v - \Pi_h^m v, w - \Pi_h^m w) + \lambda^2 s_{0, h}(v - \Pi_h^{m-j} v, w - \Pi_h^{m-j} w), \\
    \<\cA_h(v), w\> &:= \<\cA_h^c(v), w\> + \eta \<\cA_h^s(v), w\>,
\end{align*}
where $\eta > 0$ is some given stabilization factor.

\begin{remark}
    In the definition of our numerical scheme, we only need to apply the operators $\cA_h^c$, $\cA_h^s$, and $\cA_h$ to functions of $V_h$, but it will be useful in our analysis that they are also defined for functions of $V_h + \PP_{m, h}(\Omega)$.
\end{remark}

We consider the virtual element scheme (implicitly depending on $m$, $j$, $\eta$):
\begin{equation}
    \label{eq:scheme}
    \text{find } u_h \in V_h^g \text{ s.t.\ } \forall v \in V_h^0,\, \<\cA_h(u_h), v\> = 0.
\end{equation}

\section{Analysis of the scheme}
\label{sec:analysis}

In this section, we prove that the existence of a unique solution to the scheme \eqref{eq:scheme} (Theorem~\ref{thm:existence_uniqueness}) and an error estimate for this solution (Theorem~\ref{thm:error_estimate}), provided that the stabilization factor $\eta$ is large enough.

We will rely in our analysis on seminorms $|\cdot|_{*, \lambda}$ and $|\cdot|_{*, \lambda, h}$, that we define respectively on $H^2(\Omega)$ and $H^2(\Omega_h)$ by
\begin{align*}
    |v|_{*, \lambda}^2 &:= |v|_2^2 + 2 \lambda |v|_1^2 + \lambda^2 \|v\|_0^2, &
    |v|_{*, \lambda, h}^2 &:= |v|_{2, h}^2 + 2 \lambda |v|_{1, h}^2 + \lambda^2 \|v\|_0^2,
\end{align*}
noting that these are norms if $\lambda>0$.

Following the literature on nondivergence form elliptic equations with Cordes coefficients, we interpret the nonlinear nondivergence form operator $F_{\gamma, h}$ as a perturbation of the linear, divergence form operator $\cL_{\lambda, h}$. This introduces an error that can be controlled as follows.

\begin{lemma}[linear approximation]
    For any $v_1$, $v_2 \in H^2(\Omega)$,
    \begin{equation}
        \label{eq:linear_approximation}
        \|F_{\gamma, h}[v_1] - F_{\gamma, h}[v_2] - \cL_{\lambda, h} (v_1 - v_2)\|_0 \leq (1 - \kappa) |v_1 - v_2|_{*, \lambda}.
    \end{equation}
\end{lemma}

\begin{proof}
    Let $w := v_1 - v_2$. One has
    \begin{align*}
        &\|F_{\gamma, h}[v_1] - F_{\gamma, h}[v_2] - \cL_{\lambda, h} (v_1 - v_2)\|_0 \\
        &= \Big\|\inf_{\alpha_1 \in \Lambda_1} \sup_{\beta_1 \in \Lambda_2}\, [\gamma^{\alpha_1, \beta_1} (\cL_h^{\alpha_1, \beta_1} v_1 - f^{\alpha_1, \beta_1}) - \cL_{\lambda, h} v_1] \\
        &\qquad - \inf_{\alpha_2 \in \Lambda_1} \sup_{\beta_2 \in \Lambda_2}\, [\gamma^{\alpha_2, \beta_2} (\cL_h^{\alpha_2, \beta_2} v_2 - f^{\alpha_2, \beta_2}) - \cL_{\lambda, h} v_2]\Big\|_0 \\
        &\leq \Big\|\sup_{\substack{\alpha \in \Lambda_1 \\ \beta \in \Lambda_2}}\, |\gamma^{\alpha, \beta} \cL_h^{\alpha, \beta} w - \cL_{\lambda, h} w|\Big\|_0 \\
        &= \Big\|\sup_{\substack{\alpha \in \Lambda_1 \\ \beta \in \Lambda_2}}\, \big|(\gamma^{\alpha, \beta} A^{\alpha, \beta} - I_d) : \Pi_h^{m-j} \nabla^2 w + \gamma^{\alpha, \beta} b^{\alpha, \beta} \cdot \Pi_h^{m-j} \nabla w - (\gamma^{\alpha, \beta} c^{\alpha, \beta} - \lambda) \Pi_h^{m-j} w\big|\Big\|_0.
    \end{align*}
    If $\lambda = 0$, we conclude using the Cordes properties \eqref{eq:cordes_ineq_high_order_case}:
    \begin{equation*}
        \|F_{\gamma, h}[v_1] - F_{\gamma, h}[v_2] - \cL_{\lambda, h} (v_1 - v_2)\|_0
        \leq (1 - \kappa) \|\Pi_h^{m-j} \nabla^2 w\|_0
        \leq (1 - \kappa) \|\nabla^2 w\|_0
        = (1 - \kappa) |w|_{*, \lambda}.
    \end{equation*}
    If $\lambda > 0$, we use the Cauchy-Schwarz inequality and then the Cordes inequality \eqref{eq:cordes_ineq_low_order_case}:
    \begin{align*}
        &\|F_{\gamma, h}[v_1] - F_{\gamma, h}[v_2] - \cL_{\lambda, h} (v_1 - v_2)\|_0 \\
        &\leq \Big\|\sup_{\substack{\alpha \in \Lambda_1 \\ \beta \in \Lambda_2}}\, \Big(|\gamma^{\alpha, \beta} A^{\alpha, \beta} - I_d|^2 + \frac{1}{2 \lambda} |\gamma^{\alpha, \beta} b^{\alpha, \beta}|^2 + \frac{1}{\lambda^2} |\gamma^{\alpha, \beta} c^{\alpha, \beta} - \lambda|^2\Big)^{1/2} \\
        &\qquad \qquad \left(|\Pi_h^{m-j} \nabla^2 w|^2 + 2 \lambda |\Pi_h^{m-j} \nabla w|^2 + \lambda^2 |\Pi_h^{m-j} w|^2\right)^{1/2}\Big\|_0 \\
        &\leq (1 - \kappa) (\|\Pi_h^{m-j} \nabla^2 w\|_0^2 + 2 \lambda \|\Pi_h^{m-j} \nabla w\|_0^2 + \lambda^2 \|\Pi_h^{m-j} w\|_0^2)^{1/2}
        \leq (1 - \kappa) |w|_{*, \lambda},
    \end{align*}
    which concludes the proof.
\end{proof}

By similar computations, one can prove the Lipschitz continuity of the operators $F_\gamma$ and $F_{\gamma, h}$.

\begin{lemma}[Lipschitz continuity]
    For any $v_1$, $v_2 \in H^2(\Omega_h)$, one has
    \begin{align}
        \label{eq:lipschitz_continuity}
        \|F_\gamma[v_1] - F_\gamma[v_2]\|_0 &\lesssim |v_1 - v_2|_{*, \lambda, h}, &
        \|F_{\gamma, h}[v_1] - F_{\gamma, h}[v_2]\|_0 &\lesssim |v_1 - v_2|_{*, \lambda, h}.
    \end{align}
\end{lemma}

\begin{proof}
    This is easily verified, using the Cordes inequality \eqref{eq:cordes_ineq_low_order_case}, or alternatively using \eqref{eq:cordes_ineq_high_order_case} in the case $\lambda = 0$, in order to estimate $\gamma^{\alpha, \beta}(x) |A^{\alpha, \beta}(x)|$, $\gamma^{\alpha, \beta}(x) |b^{\alpha, \beta}(x)|$, and $\gamma^{\alpha, \beta}(x) c^{\alpha, \beta}(x)$ for every $\alpha \in \Lambda_1$, $\beta \in \Lambda_2$, and $x \in \Omega$.
\end{proof}

We also need the following inequality for functions of the space $V^0 = H^2(\Omega) \cap H_0^1(\Omega)$, which we refer to as a Miranda--Talenti inequality since it follows from the standard \emph{Miranda--Talenti estimate}, and reduces to it in the case $\lambda = 0$.

\begin{proposition}[Miranda--Talenti inequality]
    For any $v \in V^0$,
    \begin{equation}
        \label{eq:miranda_talenti_lambda}
        |v|_{*, \lambda} \leq \|\cL_\lambda v\|_0.
    \end{equation}
\end{proposition}

\begin{proof}
    One has
    \begin{equation*}
        \|\cL_\lambda v\|_0^2
        = \int_\Omega (\Delta v - \lambda v)^2\, d x
        = \|\Delta v\|_0^2 - 2 \lambda \int_\Omega (\Delta v) v\, d x + \lambda^2 \|v\|_0^2.
    \end{equation*}
    Integrating by parts, $\|\cL_\lambda v\|_0^2 = \|\Delta v\|_0^2 + 2 \lambda |v|_1^2 + \lambda^2 \|v\|_0^2$. Since $\Omega$ is convex and $v \in V^0$, the Miranda--Talenti estimate $|v|_2 \leq \|\Delta v\|_0$ holds (see \cite{maugeri2000,smears2013}), which concludes the proof.
\end{proof}

From the linear approximation inequality \eqref{eq:linear_approximation} and the Miranda--Talenti inequality \eqref{eq:miranda_talenti_lambda}, we deduce a strong monotonicity estimate for the operator $\cA_h$, and then the existence of a unique solution to the scheme \eqref{eq:scheme}.

\begin{lemma}[strong monotonicity]
    \label{lemma:strong_monotonicity}
    There exists $\eta_* \eqsim 1$ such that if $\eta \geq \eta_*$, it holds for any $v_1$, $v_2 \in V_h^g$ that
    \begin{equation}
        \label{eq:strong_monotonicity}
        \<\cA_h(v_1) - \cA_h(v_2), v_1 - v_2\> \geq \kappa |v_1 - v_2|_{*, \lambda}^2.
    \end{equation}
\end{lemma}

\begin{proof}
    Let $w := v_1 - v_2$. One has
    \begin{align*}
        \<\cA_h^c(v_1) - \cA_h^c(v_2), w\>
        &= \int_\Omega (F_{\gamma, h}[v_1] - F_{\gamma, h}[v_2]) \cL_{\lambda, h} w\, d x \\
        &= \|\cL_{\lambda, h} w\|_0^2 + \int_\Omega (F_{\gamma, h}[v_1] - F_{\gamma, h}[v_2] - \cL_{\lambda, h} w) \cL_{\lambda, h} w\, d x \\
        &\geq \|\cL_{\lambda, h} w\|_0^2 - \|F_{\gamma, h}[v_1] - F_{\gamma, h}[v_2] - \cL_{\lambda, h} w\|_0 \|\cL_{\lambda, h} w\|_0.
    \end{align*}
    Using \eqref{eq:linear_approximation} to estimate $\|F_{\gamma, h}[v_1] - F_{\gamma, h}[v_2] - \cL_{\lambda, h} w\|_0$, and using the fact that $\|\cL_{\lambda, h} w\|_0 = \|\Pi_h^{m-j} \cL_\lambda w\|_0 \leq \|\cL_\lambda w\|_0$,
    \begin{equation*}
        \<\cA_h^c(v_1) - \cA_h^c(v_2), w\> \geq \|\Pi_h^{m-j} \cL_\lambda w\|_0^2 - (1 - \kappa) |w|_{*, \lambda} \|\cL_\lambda w\|_0.
    \end{equation*}
    On the other hand, using the stabilization scaling inequalities \eqref{eq:stab_scaling0} and \eqref{eq:stab_scaling2} with $v = w - \Pi_h^{m} w$ and $v = w - \Pi_h^{m-j} w$,
    \begin{align*}
        \<\cA_h^s(v_1) - \cA_h^s(v_2), w\>
        &\gtrsim |w - \Pi_h^m w|_{2, h}^2 + \lambda^2 \|w - \Pi_h^{m-j} w\|_0^2 \\
        &\gtrsim \|\Delta w - \Delta \Pi_h^m w\|_0^2 + \lambda^2 \|w - \Pi_h^{m-j} w\|_0^2 \\
        &\geq \|\Delta w - \Pi_h^{m-j} \Delta w\|_0^2 + \lambda^2 \|w - \Pi_h^{m-j} w\|_0^2 \\
        &\gtrsim \|\Delta w - \Pi_h^{m-j} \Delta w - \lambda (w - \Pi_h^{m-j} w)\|_0^2 \\
        &= \|\cL_\lambda w - \Pi_h^{m-j} \cL_\lambda w\|_0^2,
    \end{align*}
    We can assume $\eta$ to be large enough so that $\eta \<\cA_h^s(v_1) - \cA_h^s(v_2), w\>
    \geq \|\cL_\lambda w - \Pi_h^{m-j} \cL_\lambda w\|_0^2$, and thus
    \begin{align*}
        \<\cA_h(v_1) - \cA_h(v_2), w\>
        &= \<\cA_h^c(v_1) - \cA_h^c(v_2), w\> + \eta \<\cA_h^s(v_1) - \cA_h^s(v_2), w\> \\
        &\geq \|\Pi_h^{m-j} \cL_\lambda w\|_0^2 - (1 - \kappa) |w|_{*, \lambda} \|\cL_\lambda w\|_0 + \|\cL_\lambda w - \Pi_h^{m-j} \cL_\lambda w\|_0^2 \\
        &= \|\cL_\lambda w\|_0^2 - (1 - \kappa) |w|_{*, \lambda} \|\cL_\lambda w\|_0,
    \end{align*}
    where the last equality follows from the Pythagorean theorem. Using two times successively the Miranda--Talenti inequality \eqref{eq:miranda_talenti_lambda}, $\|\cL_\lambda w\|_0^2 - (1 - \kappa) |w|_{*, \lambda} |\cL_\lambda w|_0 \geq \kappa \|\cL_\lambda w\|_0^2 \geq \kappa |w|_{*, \lambda}^2$, which concludes the proof.
\end{proof}

From now on, we let $\eta_* \eqsim 1$ be as in the above lemma.

\begin{theorem}[existence and uniqueness]
    \label{thm:existence_uniqueness}
    If $\eta \geq \eta_*$, then there exists a unique solution $u_h \in V_h^g$ to \eqref{eq:scheme}.
\end{theorem}

\begin{proof}
    The inequality \eqref{eq:strong_monotonicity} implies that $\cA_h$ is strongly monotone when seen as an operator from $V_h^g$ to $(V_h^0)^*$. It is easily verified that it is also Lipschitz continuous, using the Lipschitz continuity estimate \eqref{eq:lipschitz_continuity} for $F_{\gamma, h}$ and the stabilization scaling inequalities \eqref{eq:stab_scaling0} and \eqref{eq:stab_scaling2}. Therefore, the Browder--Minty theorem applies, and yields the announced result.
\end{proof}

We end this section by providing the following estimate of the error between the solution $u_h \in V_h^g$ to the scheme \eqref{eq:scheme} and an arbitrary function $u_I \in V_h^g$, which should be understood as some interpolant of the solution $u \in V^g$ to \eqref{eq:isaacs} and \eqref{eq:bc}.

\begin{theorem}[error estimate]
    \label{thm:error_estimate}
    Assume that $\eta \geq \eta_*$. Then, for any $u_I \in V_h^g$,
    \begin{equation*}
        |u_I - u_h|_{*, \lambda} \lesssim \eta \kappa^{-1} (|u - u_I|_{*, \lambda} + |u - \Pi_h^m u|_{*, \lambda, h} + \lambda E_{j,h}^0[u] + \lambda^{1/2} E_{j,h}^1[u]),
    \end{equation*}
    where
    \begin{align*}
        E_{j,h}^0[u] &:= \begin{cases}
            0 &\text{if } j = 0, \\
            \|u - \Pi_h^{m-j} u\|_0 &\text{if } j \geq 1,
        \end{cases} &
        E_{j,h}^1[u] &:= \begin{cases}
            0 &\text{if } j \leq 1, \\
            \|\nabla u - \Pi_h^{m-2} \nabla u\|_0 &\text{if } j = 2.
        \end{cases}
    \end{align*}
    In particular, if $j = 0$, only the first two bounding terms appear in the estimate, and only the first three if $j = 1$.
\end{theorem}

\begin{proof}
    Let $w := u_I - u_h$. Using the strong monotonicity inequality \eqref{eq:strong_monotonicity}, and then the fact that $u_h$ is solution to the scheme \eqref{eq:scheme},
    \begin{equation*}
        |w|_{*, \lambda}^2
        \leq \kappa^{-1} \<\cA_h(u_I) - \cA_h(u_h), w\>
        = \kappa^{-1} \<\cA_h(u_I), w\>.
    \end{equation*}
    It remains to prove that $\<\cA_h(u_I), w\> \lesssim \eta (|u - u_I|_{*, \lambda} + |u - \Pi_h^m u|_{*, \lambda, h} + \lambda E_{j,h}^0[u] + \lambda^{1/2} E_{j,h}^1[u]) |w|_{*, \lambda}$.

    One has $\<\cA_h(u_I), w\> = \<\cA_h(u_I) - \cA_h(\Pi_h^m u), w\> + \<\cA_h(\Pi_h^m u), w\>$. Using the Lipschitz continuity estimate \eqref{eq:lipschitz_continuity}, the stabilization scaling inequalities \eqref{eq:stab_scaling0} and \eqref{eq:stab_scaling2}, the triangle inequality $|u_I - \Pi_h^m u|_{*, \lambda, h} \leq |u - u_I|_{*, \lambda} + |u - \Pi_h^m u|_{*, \lambda, h}$, and the $L^2$ stability of the projectors, it is easy to prove that
    \begin{equation*}
        \<\cA_h(u_I) - \cA_h(\Pi_h^m u), w\> \lesssim \eta (|u - u_I|_{*, \lambda} + |u - \Pi_h^m u|_{*, \lambda, h}) |w|_{*, \lambda}.
    \end{equation*}

It remains to bound $\<\cA_h(\Pi_h^m u), w\>=\<\cA_h^c(\Pi_h^m u), w\>+\<\cA_h^s(\Pi_h^m u), w\>$.
%
Since $\Pi_h^m u \in \PP_{m, h}(\Omega)$, one has $\<\cA_h^s(\Pi_h^m u), w\> = 0$ if $j=0$ and otherwise
    \begin{equation*}
        \<\cA_h^{s}(\Pi_h^m u), w\>
        = \eta \lambda^2 s_{0, h}(\Pi_h^m u - \Pi_h^{m-j} u, w - \Pi_h^{m-j} w)
        \lesssim \eta \lambda \|u - \Pi_h^{m-j} u\|_0 |w|_{*, \lambda}
        = \eta \lambda E_{j, h}^0[u] |w|_{*, \lambda}.
    \end{equation*}
On the other hand,
    \begin{align*}
        \<\cA_h^{c}(\Pi_h^m u), w\>
        &= \int_\Omega F_{\gamma, h}[\Pi_h^m u] \cL_{\lambda, h} w\, d x \\
        &= \int_\Omega F_\gamma[\Pi_h^m u] \cL_{\lambda, h} w\, d x + \int_\Omega (F_{\gamma, h}[\Pi_h^m u] - F_\gamma[\Pi_h^m u]) \cL_{\lambda, h} w\, d x.
    \end{align*}
   Since $u$ is a strong solution to \eqref{eq:isaacs}, one has $F_\gamma[u] = 0$ almost everywhere, so
    \begin{equation*}
        \int_\Omega F_\gamma[\Pi_h^m u] \cL_{\lambda, h} w\, d x
        = \int_\Omega (F_\gamma[\Pi_h^m u] - F_\gamma[u]) \cL_{\lambda, h} w\, d x
        \lesssim |u - \Pi_h^m u|_{*, \lambda, h} |w|_{*, \lambda},
    \end{equation*}
    using the Lipschitz continuity estimate \eqref{eq:lipschitz_continuity} for $F_\gamma$. It remains to estimate the term
    \begin{equation*}
        \int_\Omega (F_{\gamma, h}[\Pi_h^m u] - F_\gamma[\Pi_h^m u]) \cL_{\lambda, h} w\, d x.
    \end{equation*}
    One has
    \begin{align*}
        \|F_{\gamma, h}[\Pi_h^m u] - F_\gamma[\Pi_h^m u]\|_0
        &\leq \Big\|\sup_{\substack{\alpha \in \Lambda_1 \\ \beta \in \Lambda_2}} |\gamma^{\alpha, \beta} (\cL_h^{\alpha, \beta} \Pi_h^m u - \cL^{\alpha, \beta} \Pi_h^m u)|\Big\|_0 \\
        &\leq \Big\|\sup_{\substack{\alpha \in \Lambda_1 \\ \beta \in \Lambda_2}}\,\big|\gamma^{\alpha, \beta}b^{\alpha, \beta} \cdot (\Pi_h^{m-j}-I) \nabla \Pi_h^m u \big|\Big\|_0 \\
        &\quad + \Big\|\sup_{\substack{\alpha \in \Lambda_1 \\ \beta \in \Lambda_2}} |\gamma^{\alpha, \beta} c^{\alpha, \beta} (\Pi_h^{m-j} - I) \Pi_h^m u|\Big\|_0 \\
       &\lesssim \lambda^{1/2} \|(\Pi_h^{m-j} - I) \nabla \Pi_h^m u\|_0 + \lambda \|\Pi_h^{m-j} u - \Pi_h^m u\|_0,
    \end{align*}
    where we controlled $|\gamma^{\alpha, \beta}b^{\alpha, \beta}|$ and  $|\gamma^{\alpha, \beta} c^{\alpha, \beta}|$ using \eqref{eq:cordes_ineq_low_order_case}. Observe that
    \begin{align*}
        \|\Pi_h^{m-j} u - \Pi_h^m u\|_0 &\leq E_{j, h}^0[u], &
        \lambda^{1/2} \|(\Pi_h^{m-j} - I) \nabla \Pi_h^m u\|_0 \leq \lambda^{1/2} E_{j, h}^1[u] + |u - \Pi_h^m u|_{*, \lambda, h},
    \end{align*}
    where for the second inequality we used that the only nontrivial case is $j=2$, in which
    \begin{align*}
        \|(\Pi_h^{m-2} - I) \nabla \Pi_h^m u\|_0
        &\leq \|(\Pi_h^{m-2} - I) \nabla u\|_0 + \|(\Pi_h^{m-2} - I) \nabla (u - \Pi_h^m u)\|_0 \\
        &\leq E_{2, h}^1[u] + \|\nabla (u - \Pi_h^m u)\|_0.
    \end{align*}
    This concludes the proof.
\end{proof}

\section{Variants of the scheme}
\label{sec:variants}

In this section, we analyze two variants of the scheme. The first considers the use of different polynomial projection orders for the various terms in the differential operator, following the approach of~\cite{cangiani2017}. While equal-order projections simplify the analysis, it is natural to ask whether projection orders tailored to the differential terms can also be employed. As shown below, an a priori error analysis can still be established; however, the resulting bounds exhibit a stronger dependence on the Cordes condition, involving an additional factor of $\kappa^{-1}$.
The second variant concerns the weak imposition of boundary conditions. In dimensions $d\in\{2,3\}$ (to which we restrict our analysis for technical reasons), we show that it provides a natural alternative to strong imposition, yielding an equally optimal method without additional parameters and with a simpler implementation.

\subsection{Different polynomial projection degrees}
\label{subsec:scheme_alt1}

Let us define the following variants of the operators $\cL_h^{\alpha, \beta}$, $\cL_{\lambda, h}$, and $F_{\gamma, h}$ from \eqref{eq:projected_pointwise_operators}:
\begin{align*}
    \widetilde \cL_h^{\alpha, \beta} v &:= A^{\alpha, \beta} : \Pi_h^{m-2} \nabla^2 v + b^{\alpha, \beta} \cdot \Pi_h^{m-1} \nabla v - c^{\alpha, \beta} \Pi_h^m v, \\
    \widetilde \cL_{\lambda, h} v &:= \Pi_h^{m-2} \Delta v - \lambda \Pi_h^m v, \\
    \widetilde F_{\gamma, h}[v] &:= \inf_{\alpha \in \Lambda_1} \sup_{\beta \in \Lambda_2}\, [\gamma^{\alpha, \beta} (\widetilde \cL_h^{\alpha, \beta} v - f^{\alpha, \beta})].
\end{align*}
We then define, for $v \in V_h + \PP_{m, h}(\Omega)$ and $w \in V_h$,
\begin{align*}
    \<\widetilde \cA_h^c(v), w\> &:= \int_\Omega \widetilde F_{\gamma, h}[v] \widetilde \cL_{\lambda, h} w\, d x, \\
    \<\widetilde \cA_h^s(v), w\> &:= s_{2, h}(v - \Pi_h^m v, w - \Pi_h^m w) + \lambda^2 s_{0, h}(v - \Pi_h^m v, w - \Pi_h^m w), \\
    \<\widetilde \cA_h(v), w\> &:= \<\widetilde \cA_h^c(v), w\> + \eta \kappa^{-1} \<\widetilde \cA_h^s(v), w\>,
\end{align*}
and consider the virtual element scheme:
\begin{equation}
    \label{eq:scheme_alt1}
    \text{find } \widetilde u_h \in V_h^g \text{ s.t.\ } \forall v \in V_h^0,\, \<\widetilde \cA_h(\widetilde u_h), v\> = 0.
\end{equation}

As in section~\ref{sec:scheme}, $\eta > 0$ is a stabilization factor, that should be understood as a parameter of the scheme. Observe however that in the definition of $\<\widetilde \cA_h(v), w\>$, the term $\<\widetilde \cA_h^s(v), w\>$ is multiplied by the factor $\eta \kappa^{-1}$, while in the definition of $\<\cA_h(v), w\>$ in section~\ref{sec:scheme}, the term $\<\cA_h^s(v), w\>$ was only multiplied by $\eta$. The additional factor $\kappa^{-1}$ in the construction of the scheme \eqref{eq:scheme_alt1} is required in the proof of the following strong monotonicity property, which reflects Lemma~\ref{lemma:strong_monotonicity} in the new setting.

\begin{lemma}
    \label{lemma:strong_monotonicity_alt1}
    There exists $\widetilde \eta_* \eqsim 1$ such that if $\eta \geq \widetilde \eta_*$, it holds for any $v_1$, $v_2 \in V_h^g$ that
    \begin{equation}
        \label{eq:strong_monotonicity_alt1}
        \<\widetilde \cA_h(v_1) - \widetilde \cA_h(v_2), v_1 - v_2\> \gtrsim \kappa |v_1 - v_2|_{*, \lambda}^2.
    \end{equation}
\end{lemma}

\begin{proof}
    Proceeding as in the proof of Lemma~\ref{lemma:strong_monotonicity}, for $v_1$, $v_2 \in V_h^g$ and $w = v_1 - v_2$,
    \begin{equation*}
        \<\widetilde \cA_h^c(v_1) - \widetilde \cA_h^c(v_2), w\> \geq \|\widetilde \cL_{\lambda, h} w\|_0^2 - (1 - \kappa) |w|_{*, \lambda} \|\widetilde \cL_{\lambda, h} w\|_0.
    \end{equation*}
    Using the Miranda--Talenti estimate \eqref{eq:miranda_talenti_lambda},
    \begin{equation*}
        \<\widetilde \cA_h^c(v_1) - \widetilde \cA_h^c(v_2), w\> \geq \|\widetilde \cL_{\lambda, h} w\|_0^2 - (1 - \kappa) \|\cL_\lambda w\|_0 \|\widetilde \cL_{\lambda, h} w\|_0.
    \end{equation*}
    Using that $\widetilde \cL_{\lambda, h} w = \Pi_h^m \cL_\lambda w - (\Pi_h^m \Delta w - \Pi_h^{m-2} \Delta w)$ and
    \begin{align*}
        \|\Pi_h^m \cL_\lambda w\|_0 &\leq \|\cL_\lambda w\|_0, &
        \|\Pi_h^m \Delta w - \Pi_h^{m-2} \Delta w\|_0 &\leq \|\Delta w - \Pi_h^{m-2} \Delta w\|_0,
    \end{align*}
    one has
    \begin{align*}
        \|\widetilde \cL_{\lambda, h} w\|_0^2
        &= \|\Pi_h^m \cL_\lambda w\|_0^2 + \|\Pi_h^m \Delta w - \Pi_h^{m-2} \Delta w\|_0^2 - 2 \int_\Omega \Pi_h^m \cL_\lambda w (\Pi_h^m \Delta w - \Pi_h^{m-2} \Delta w)\, d x \\
        &\geq \|\Pi_h^m \cL_\lambda w\|_0^2 - 2 \int_\Omega \Pi_h^m \cL_\lambda w (\Pi_h^m \Delta w - \Pi_h^{m-2} \Delta w)\, d x \\
        &\geq \left(1 - \frac{\kappa}{3}\right) \|\Pi_h^m \cL_\lambda w\|_0^2 - \frac{3}{\kappa} \|\Delta w - \Pi_h^{m-2} \Delta w\|_0^2
    \end{align*}
    and
    \begin{align*}
        (1 - \kappa) \|\cL_\lambda w\|_0 \|\widetilde \cL_{\lambda, h} w\|_0 &\leq (1 - \kappa) \|\cL_\lambda w\|_0^2 + (1 - \kappa) \|\cL_\lambda w\|_0 \|\Delta w - \Pi_h^{m-2} \Delta w\|_0 \\
        &\leq \left(1 - \kappa + \frac{\kappa}{3}\right) \|\cL_\lambda w\|_0^2 + \frac{3 (1 - \kappa)^2}{4 \kappa} \|\Delta w - \Pi_h^{m-2} \Delta w\|_0^2 \\
        &\leq \left(1 - \frac{2 \kappa}{3}\right) \|\cL_\lambda w\|_0^2 + \frac{3}{4 \kappa} \|\Delta w - \Pi_h^{m-2} \Delta w\|_0^2,
    \end{align*}
    hence
    \begin{equation*}
        \<\widetilde \cA_h^c(v_1) - \widetilde \cA_h^c(v_2), w\> \geq \left(1 - \frac{\kappa}{3}\right) \|\Pi_h^m \cL_\lambda w\|_0^2 - \left(1 - \frac{2 \kappa}{3}\right) \|\cL_\lambda w\|_0^2 - \frac{15}{4 \kappa} \|\Delta w - \Pi_h^{m-2} \Delta w\|_0^2.
    \end{equation*}
    One may conclude by proceeding similarly to the proof of Lemma~\ref{lemma:strong_monotonicity}. In particular, the factor $\kappa^{-1}$ in front of the stabilization term ensures that the term $-\frac{15}{4 \kappa} \|\Delta w - \Pi_h^{m-2} \Delta w\|_0^2$ in the above estimate is compensated.
\end{proof}

From now on, we let $\widetilde \eta_*$ be as in the above lemma. As in section~\ref{sec:analysis}, we deduce by applying the Browder--Minty theorem that the scheme \eqref{eq:scheme_alt1} admits a unique solution $\widetilde u_h \in V_h$ as soon as is $\eta \geq \widetilde \eta_*$. The following error estimate holds.

\begin{theorem}
    \label{thm:error_estimate_alt1}
    We assume that $\eta \geq \widetilde \eta_*$. Then, for any $u_I \in V_h^g$,
    \begin{equation*}
        |u_I - \widetilde u_h|_{*, \lambda} \lesssim \eta \kappa^{-2} (|u - u_I|_{*, \lambda} + |u - \Pi_h^m u|_{*, \lambda, h}).
    \end{equation*}
\end{theorem}

\begin{proof}
    The proof is a straightforward adaptation of the one of Theorem~\ref{thm:error_estimate}. The main difference between the statements of Theorem~\ref{thm:error_estimate_alt1} and Theorem~\ref{thm:error_estimate} is the exponent on $\kappa$. Here, one negative power of $\kappa$ arises from the strong monotonicity estimate \eqref{eq:strong_monotonicity_alt1}, similarly to the case of Theorem~\ref{thm:error_estimate}, and one additional negative power of $\kappa$ arises when upper bounding the term $\eta \kappa^{-1} \<\widetilde \cA_h^s(u_I) - \widetilde \cA_h^s(\Pi_h^m u), u_I - u_h\>$.
\end{proof}

\subsection{Weak imposition of the boundary condition}
\label{subsec:nitsche}

We now consider a variant of the scheme \eqref{eq:scheme} in which the Dirichlet boundary condition \eqref{eq:bc} is imposed weakly, through a Nitsche-type boundary penalty method. Here, we consider only the cases $d=2$ and $d=3$.

We denote by $\cS$ the set of sides of $\Omega$, and, for every $S \in \cS$, by $\Sigma_S$ the set of ridges of $\Omega$ (vertices if $d=2$ and edges if $d=3$) which are also sides of  $S$.

We define the space
\begin{equation}
    \label{eq:boundary_space}
    V_{\partial \Omega} := \left\{v \in \prod_{S \in \cS} H^{3/2}(S) \mid \forall S_1, S_2 \in \cS,\, \forall \sigma \in \Sigma_{S_1} \cap \Sigma_{S_2},\, (v_{S_1})_{|\sigma} = (v_{S_2})_{|\sigma}\right\},
\end{equation}
equipped with the norm $\|\cdot\|_{V_{\partial \Omega}}$ defined by
\begin{equation*}
    \|v\|_{V_{\partial \Omega}}^2 := \sum_{S \in \cS} \|v\|_{H^{3/2}(S)}^2.
\end{equation*}

According to \cite{bernardi2007}, the space
$V_{\partial \Omega}$ allows for the construction of a continuous trace operator for $H^2(\Omega)$. Note that  only the value trace is considered, which is nonstandard in the $H^2$ setting.

\begin{theorem}
    \label{thm:extension}
    The trace operator
    \begin{equation*}
        T \colon H^2(\Omega) \to \prod_{S \in \cS} H^{3/2}(S), \quad v \mapsto (v_{|S})_{S \in \cS},
    \end{equation*}
%
    is continuous and onto the space $V_{\partial \Omega}$, and admits a continuous right inverse $E \colon V_{\partial \Omega} \to H^2(\Omega)$. 
\end{theorem}
\begin{proof}
The assertion is a special case of the general results established in~\cite[Chapter~I]{bernardi2007} given by Corollary~5.8 and Corollary~6.10 in the $d=2$ and $d=3$ case, respectively. Specifically, it follows by verifying  the conditions in  \cite[Theorem 5.7]{bernardi2007} and \cite[Theorem 6.9]{bernardi2007}.

In~\cite{bernardi2007}, the general case of traces from $\mathbb{W}^{{\rm s,p(m)}}(\partial\Omega) := \prod_{S \in \cS} \prod_{k=0}^{{\rm m} - 1} W^{{\rm s} - k - \frac{1}{\rm p}, {\rm p}}(S)$ is considered, where $1<{\rm p}<\infty$, ${\rm s}>1/{\rm p}$, and  ${\rm m}$ denotes the number of traces and has to satisfy $0\le {\rm m}-1<{\rm s}-1/{\rm p}$. In our setting, ${\rm m}=1$, ${\rm s}=2$, and ${\rm p}=2$.

We deal with $d=2$ first, verifying the conditions required by \cite[Theorem~5.7]{bernardi2007}. In this case, the only matching condition required is continuity at the vertices which is already  imposed in the definition of $V_{\partial \Omega}$. This corresponds to \cite[(5.7)]{bernardi2007} with ${\rm n}=0<{\rm s}-2/{\rm p}=1$ and is meaningful as $H^{3/2}(S)$ is (comfortably) continuously embedded into the space of functions which are continuous in $S$. Condition \cite[(5.8)]{bernardi2007} is void when ${\rm m}= 1$, because, according to \cite[(5.1)]{bernardi2007}, the operators $\cL_-$ and $\cL_+$ are proportional to the first order partial derivation operators in the respective, non colinear directions $\tau_{\ell_-(i)}$ and $\tau_{\ell_+(i)}$, and then the condition $\cL_- + \cL_+ = 0$ imposes that they both coincide with the zero operator. This is reflected by the fact that no condition appears for the case ${\rm m} = 1$ and ${\rm n} = 2$ in \cite[Table~5.1]{bernardi2007}.

In the case $d=3$, we need to verify the conditions in \cite[Theorem~6.9]{bernardi2007}. The matching conditions (i) in \cite[Theorem~6.9]{bernardi2007} correspond on edges to the conditions \cite[(5.7)]{bernardi2007} on vertices already discussed in the $d=2$ case. These reduce to the case ${\rm n}=0<{\rm s}-2/{\rm p}=1$ imposing the continuity at each edge $\sigma$ as required in the definition of $V_{\partial \Omega}$. Note that, also in this case, as $v\in H^{3/2}(S)$ implying $v|_\sigma\in H^1(\sigma)$ for all ridges (edges) $\sigma \in \partial S$, we have that the trace $v|_\sigma$ can be selected as a  continuous function. As before, there are no extra conditions on edges as the differential operators $\cL_-$ and $\cL_+$ involved in \cite[(6.4)]{bernardi2007} are zero. We further need to verify conditions (ii) in \cite[Theorem~6.9]{bernardi2007} which pertain to vertices. In particular, we have only the matching condition corresponding to ${\rm n}=0<{\rm s}-3/{\rm p}=1/2$, requiring continuity at each vertex. However, this follows from the fact that edge traces are compatible in the sense of continuous functions.
\end{proof}

Let us define the set
\begin{equation*}
    \cS_{\partial, h} := \{s \text{ side of } K \mid K \in \cK_h,\, s \subset \partial \Omega\}
\end{equation*}
of boundary sides of the mesh $\cK_h$, the space $\PP_{m, h}(\partial \Omega)$ of functions $p \in L^2(\partial \Omega)$ satisfying $p_{|s} \in \PP_m(s)$ on each $s \in \cS_{\partial, h}$, the $L^2$ projection operator $\Pi_{\partial, h}^m$ from $L^2(\partial \Omega)$ to $\PP_{m, h}(\partial \Omega)$, and the space $V_{\partial, h} := \{v_{|\partial \Omega} \mid v \in V_h\}$. For any $v \in L^2(\partial \Omega)$, let
\begin{equation*}
    |v|_{\partial, \lambda, h}^2 := \sum_{s \in \cS_{\partial, h}} (h_s^{-3} + \lambda h_s^{-1}) \|v\|_{0, s}^2.
\end{equation*}

\begin{proposition}\label{prop:lift}
For any $v \in V_h$, there exists $v' \in H^2(\Omega)$ such that $v'|_{\partial \Omega} = v$ and $|v'|_2^2 \lesssim  |v|_{\partial, 0, h}^2$.
\end{proposition}
\begin{proof}
By Theorem~\ref{thm:extension}, there exists $v' \in H^2(\Omega)$ such that $v' = v$ on $\partial \Omega$ and $|v'|_2^2 \lesssim  \|v\|_{V_{\partial \Omega}}^2$, and it remains to localize the contribution from each side $S \in \cS$. We exploit the fractional a norm localization result~\cite[Equation (3.2)]{borthagaray2021}, see also \cite{faermann2002} for the original proof. Let $S \in \cS$ and for any $s \in \cS_{\partial, h}$ with $s \subset S$, denote by $\omega_s:=\cup \{s' \in \cS_{\partial, h} \mid s' \subset S,\, s'\cap s\neq \emptyset \}$ the patch of coplanar boundary sides surrounding $s$. Then, from~\cite[Equation (3.2)]{borthagaray2021}, the shape-regularity Assumption~\ref{ass:reg}, and a Sobolev interpolation inequality~\cite{adams2003}, we obtain
\begin{align*}
|v|_{H^{3/2}(S)}^2 &\lesssim\sum_{s \in \cS_{\partial, h},\, s \subset S}
\left(|v|_{3/2,\omega_s}^2+h_s^{-1}|\nabla v|_{0,s}^2\right)\\
&\lesssim \sum_{s \in \cS_{\partial, h},\, s \subset S}
\left(h_s|v|_{2,s}^2+h_s^{-1}|\nabla v|_{0,s}^2\right)
\end{align*}
depending only on $(d-1)$ and on shape-regularity. Thus, using the inverse inequalities (Assumption~\ref{assum:inverse_ineq}), we conclude
\begin{align*}
 \|v\|_{V_{\partial \Omega}}^2&=\sum_{S \in \cS}\|v\|_{H^{3/2}(S)}^2 \lesssim
\sum_{s \in \cS_{\partial, h}}\left(h_s|v|_{2,s}^2+h_s^{-1}|\nabla v|_{0,s}^2+\|v\|_{1,s}^2\right)
\lesssim \sum_{s \in \cS_{\partial, h}} h_s^{-3}\|v\|_{0,s}.
\end{align*}

\end{proof}

We assume that we are given a symmetric bilinear form $s_{\partial, \lambda, h} \colon (V_{\partial, h} + \PP_{m, h}(\partial \Omega))^2 \to \RR$ satisfying the following assumption.

\begin{assumption}[boundary stabilization form scaling]
    For any $v \in V_{\partial, h} + \PP_{m, h}(\partial \Omega)$, one has
    \begin{equation}
        \label{eq:boundary_stab_scaling}
        s_{\partial, \lambda, h}(v, v) \eqsim |v|_{\partial, \lambda, h}^2.
    \end{equation}
\end{assumption}

Let us define operators $\cA_h^b$ and $\cA_h^w$ from $V_h + \PP_{m, h}(\Omega)$ to $(V_h)^*$ by
\begin{align*}
    \<\cA_h^b(v), w\> &:= \sum_{s \in \cS_{\partial, h}} (h_s^{-3} + \lambda h_s^{-1}) \int_s (\Pi_s^m v - g) \Pi_s^m w\, d x + s_{\partial, \lambda, h}(v - \Pi_{\partial, h}^m v, w - \Pi_{\partial, h}^m w), \\
    \<\cA_h^w(v), w\> &:= \<\cA_h(v), w\> + \eta \kappa^{-1} \<\cA_h^b(v), w\>.
\end{align*}
We consider the scheme
\begin{equation}
    \label{eq:nitsche_scheme}
    \text{find } u_h^w \in V_h \text{ s.t.\ } \forall v \in V_h,\, \<\cA_h^w(u_h^w), v\> = 0.
\end{equation}

In our analysis of the scheme \eqref{eq:nitsche_scheme}, we rely on the following generalization of the Miranda--Talenti inequality \eqref{eq:miranda_talenti_lambda} to functions of $V_h$, that do not necessarily belong to $H_0^1(\Omega)$.

\begin{proposition}[Miranda--Talenti inequality, boundary penalty case]
    For any $v \in V_h$,
    \begin{equation}
        \label{eq:nitsche_miranda_talenti_lambda}
        |v|_{*, \lambda}^2 - \|\cL_\lambda v\|_0^2 \lesssim |v|_{\partial, \lambda, h} |v|_{*, \lambda}.
    \end{equation}
\end{proposition}

\begin{proof}
    We distinguish between the cases $\lambda = 0$ and $\lambda > 0$.

    \emph{Case $\lambda = 0$.}
By Proposition~\ref{prop:lift}  there exists $v' \in H^2(\Omega)$ such that $v' = v$ on $\partial \Omega$ and $|v'|_2 \lesssim |v|_{\partial, 0, h}^2 $.

    Let $v_0 := v - v'$, so that $|v|_2 \leq |v_0|_2 + |v'|_2$. By the standard Miranda--Talenti estimate and then the triangle inequality, $|v_0|_2 \leq \|\Delta v_0\|_0 \leq \|\Delta v\|_0 + \|\Delta v'\|_0$, so $|v|_2 \leq \|\Delta v\|_0 + \|\Delta v'\|_0 + |v'|_2$, thus $|v|_2 - \|\Delta v\|_0 \lesssim |v'|_2 \lesssim |v|_{\partial, 0, h}$ and $|v|_2^2 - \|\Delta v\|_0^2 = (|v|_2 - \|\Delta v\|_0) (|v|_2 + \|\Delta v\|_0) \lesssim |v|_{\partial, 0, h} |v|_2$.

    \emph{Case $\lambda > 0$.} One has
    \begin{equation*}
        \|\cL_\lambda v\|_0^2
        = \|\Delta v\|_0^2 - 2 \lambda \int_\Omega v \Delta v\, d x + \lambda^2 \|v\|_0^2
        = \|\Delta v\|_0^2 + 2 \lambda |v|_1^2 + \lambda^2 \|v\|_0^2 - 2 \lambda \int_{\partial \Omega} v \nabla v \cdot \bn\, d x,
    \end{equation*}
    so that
    \begin{equation*}
        |v|_{*, \lambda}^2 - \|\cL_\lambda v\|_0^2
        = |v|_2^2 - \|\Delta v\|_0^2 + 2 \lambda \int_{\partial \Omega} v \nabla v \cdot \bn\, d x.
    \end{equation*}
    From the analysis of the case $\lambda = 0$, $|v|_2^2 - \|\Delta v\|_0^2 \lesssim |v|_{\partial, 0, h} |v|_2 \leq |v|_{\partial, \lambda, h} |v|_{*, \lambda}$. On the other hand,
    \begin{equation*}
        2 \lambda \int_{\partial \Omega} v \nabla v \cdot \bn\, d x
        \lesssim \left(\lambda \sum_{s \in \cS_{\partial, h}} h_s^{-1} \|v\|_{0, s}^2\right)^{1/2} \left(\lambda \sum_{s \in \cS_{\partial, h}} h_s \|\nabla v\|_{0, s}^2\right)^{1/2}.
    \end{equation*}
    For any $s \in \cS_{\partial, h}$, by a standard trace inequality and the inverse inequality \eqref{eq:inverse_ineq_cell}, one has $h_s \|\nabla v\|_{0, s}^2 \lesssim \|\nabla v\|_{0, K}^2$, where $K$ denotes the cell of $\cK_h$ whose $s$ is a side. We deduce that $2 \lambda \int_{\partial \Omega} v \nabla v \cdot \bn\, d x \lesssim |v|_{\partial, \lambda, h} |v|_{*, \lambda}$, which concludes the proof.
\end{proof}

The new Miranda--Talenti inequality \eqref{eq:nitsche_miranda_talenti_lambda} allows us to prove a suitable strong monotonicity estimate for the operator $\cA_h^w$.

\begin{lemma}[strong monotonicity, boundary penalty case]
    \label{lemma:nitsche_strong_monotonicity}
    There exists $\eta_*^w \eqsim 1$ such that if $\eta \geq \eta_*^w$, it holds for any $v_1$, $v_2 \in V_h$ that
    \begin{equation}
        \label{eq:nitsche_strong_monotonicity}
        \<\cA_h^w(v_1) - \cA_h^w(v_2), v_1 - v_2\> \gtrsim \kappa |v_1 - v_2|_{*, \lambda}^2 + \kappa^{-1} |v_1 - v_2|_{\partial, \lambda, h}^2.
    \end{equation}
\end{lemma}

\begin{proof}
    Let $w := v_1 - v_2$. Proceeding as in the proof of Lemma~\ref{lemma:strong_monotonicity}, one has
    \begin{equation*}
        \<\cA_h(v_1) - \cA_h(v_2), w\> \geq \|\cL_\lambda w\|_0^2 - (1 - \kappa) |w|_{*, \lambda} \|\cL_\lambda w\|_0.
    \end{equation*}
    Similarly, it is easily proved that $\<\cA_h^b(v_1) - \cA_h^b(v_2), w\> \gtrsim |w|_{\partial, \lambda, h}^2$.

    Let $C$ denote the hidden constant in \eqref{eq:nitsche_miranda_talenti_lambda}, so that $|w|_{*, \lambda}^2 - \|\cL_\lambda w\|_0^2 \leq C |w|_{\partial, \lambda, h} |w|_{*, \lambda}$. We can assume $\eta$ to be large enough so that $\eta \kappa^{-1} \<\cA_h^b(v_1) - \cA_h^b(v_2), w\> \geq C^2 \kappa^{-1} |w|_{\partial, \lambda, h}^2$. Then
    \begin{equation*}
        \<\cA_h^w(v_1) - \cA_h^w(v_2), w\> \geq \|\cL_\lambda w\|_0^2 - (1 - \kappa) |w|_{*, \lambda} \|\cL_\lambda w\|_0 + C^2 \kappa^{-1} |w|_{\partial, \lambda, h}^2,
    \end{equation*}
    and, using Young's inequality,
    \begin{align*}
        \<\cA_h^w(v_1) - \cA_h^w(v_2), w\>
        &\geq \kappa \|\cL_\lambda w\|_0^2 - \frac{1 - \kappa}{2} (|w|_{*, \lambda}^2 - \|\cL_\lambda w\|_0^2) + C^2 \kappa^{-1} |w|_{\partial, \lambda, h}^2 \\
        &\geq \kappa \|\cL_\lambda w\|_0^2 - \frac{C}{2} |w|_{\partial, \lambda, h} |w|_{*, \lambda} + C^2 \kappa^{-1} |w|_{\partial, \lambda, h}^2.
    \end{align*}
    Starting with the inequality $|w|_{*, \lambda}^2 - \|\cL_\lambda w\|_0^2 \leq C |w|_{\partial, \lambda, h} |w|_{*, \lambda}$ and applying Young's inequality, one has $|w|_{*, \lambda}^2 - \|\cL_\lambda w\|_0^2 \leq (C^2/2) |w|_{\partial, \lambda, h}^2 + (1/2) |w|_{*, \lambda}^2$, so $\|\cL_\lambda w\|_0^2 + (C^2/2) |w|_{\partial, \lambda, h}^2 \geq (1/2) |w|_{*, \lambda}^2$ and
    \begin{equation*}
        \<\cA_h^w(v_1) - \cA_h^w(v_2), w\>
        \geq \frac{1}{2} \kappa |w|_{*, \lambda}^2 - \frac{C}{2} |w|_{\partial, \lambda, h} |w|_{*, \lambda} + \frac{C^2}{2} \kappa^{-1} |w|_{\partial, \lambda, h}^2.
    \end{equation*}
    We conclude using Young's inequality one more time.
\end{proof}

From now on, we let $\eta_*^w$ be as in the above lemma. As in section~\ref{sec:analysis}, we deduce by applying the Browder--Minty theorem that the scheme \eqref{eq:nitsche_scheme} admits a unique solution $u_h^w \in V_h$ as soon as is $\eta \geq \eta_*^w$. Adapting the proof of Theorem~\ref{thm:error_estimate}, we obtain the following error estimate for $u_h^w$.

\begin{theorem}[error estimate, boundary penalty case]
    \label{thm:nitsche_error_estimate}
    We assume that $\eta \geq \eta_*^w$. Then, for any $u_I \in V_h$,
    \begin{align*}
        |u_I - u_h^w|_{*, \lambda} + \kappa^{-1} |u_I - u_h^w|_{\partial, \lambda}
        &\lesssim \eta \kappa^{-1} (|u - u_I|_{*, \lambda} + |u - u_I|_{\partial, \lambda, h} + |u - \Pi_h^m u|_{*, \lambda, h} \\
        &\qquad \qquad + \lambda E_{j, h}^0[u] + \lambda^{1/2} E_{j, h}^1[u]),
    \end{align*}
    where $E_{j, h}^0[u]$ and $E_{j, h}^1[u]$ are as in Theorem~\ref{thm:error_estimate}.
\end{theorem}

\begin{proof}
    Let $w := u_I - u_h^w$. Using the strong monotonicity estimate \eqref{eq:nitsche_strong_monotonicity} and then the fact that $u_h^w$ is solution to the scheme \eqref{eq:nitsche_scheme},
    \begin{equation*}
        (|w|_{*, \lambda} + \kappa^{-1} |w|_{\partial, \lambda, h})^2
        \lesssim \kappa^{-1} \<\cA_h^w(u_I) - \cA_h^w(u_h^w), w\>
        = \kappa^{-1} \<\cA_h^w(u_I), w\>.
    \end{equation*}
    It remains to prove a suitable upper bound on $\<\cA_h^w(u_I), w\>$.

    Recall that $\<\cA_h^w(u_I), w\> = \<\cA_h(u_I), w\> + \eta \kappa^{-1} \<\cA_h^b(u_I), w\>$. For the term $\<\cA_h(u_I), w\>$, one may proceed as in Theorem~\ref{thm:error_estimate}. Let us prove, using similar arguments, that $\<\cA_h^b(u_I), w\> \lesssim (|u - u_I|_{\partial, \lambda, h} + |u - \Pi_h^m u|_{*, \lambda, h}) |w|_{\partial, \lambda, h}$.

    One has $\<\cA_h^b(u_I), w\> = \<\cA_h^b(u_I) - \cA_h^b(\Pi_h^m u), w\> + \<\cA_h^b(\Pi_h^m u), w\>$. Using the stabilization scaling inequality \eqref{eq:boundary_stab_scaling} and the triangle inequality,
    \begin{equation*}
        \<\cA_h^b(u_I) - \cA_h^b(\Pi_h^m u), w\> \lesssim (|u - u_I|_{\partial, \lambda, h} + |u - \Pi_h^m u|_{\partial, \lambda, h}) |w|_{\partial, \lambda, h}.
    \end{equation*}
    By standard trace inequalities and then standard polynomial projection estimates, $|u - \Pi_h^m u|_{\partial, \lambda, h} \lesssim (h^{-2} + \lambda^{1/2} h^{-1}) \|u - \Pi_h^m u\|_{0, h} + (h^{-1} + \lambda^{1/2}) |u - \Pi_h^m u|_{1, h} \lesssim |u - \Pi_h^m u|_{*, \lambda, h}$, and thus
    \begin{equation*}
        \<\cA_h^b(u_I) - \cA_h^b(\Pi_h^m u), w\> \lesssim (|u - u_I|_{\partial, \lambda, h} + |u - \Pi_h^m u|_{*, \lambda, h}) |w|_{\partial, \lambda, h}.
    \end{equation*}
    It remains to bound the term $\<\cA_h^b(\Pi_h^m u), w\>$.

    Using that $\Pi_h^m u \in \PP_{m, h}(\Omega)$, one can compute that
    \begin{equation*}
        \<\cA_h^b(\Pi_h^m u), w\>
        = \sum_{s \in \cS_{\partial, h}} (h_s^{-3} + \lambda h_s^{-1}) \int_s (\Pi_h^m u - g) \Pi_s^m w\, d x,
    \end{equation*}
    and thus $\<\cA_h^b(\Pi_h^m u), w\> \lesssim |\Pi_h^m u - g|_{\partial, \lambda, h} |w|_{\partial, \lambda, h}$. Using that $u = g$ a.e.\ on $\partial \Omega$, $\<\cA_h^b(\Pi_h^m u), w\> \lesssim |u - \Pi_h^m u|_{\partial, \lambda, h} |w|_{\partial, \lambda, h}$. As above, $|u - \Pi_h^m u|_{\partial, \lambda, h} \lesssim |u - \Pi_h^m u|_{*, \lambda, h}$, which concludes the proof.
\end{proof}

\section{Realization of the VEM framework}\label{sec:realvem}

In this section, we briefly outline an $H^2$-conforming virtual element realization of the framework of section~\ref{sec:vem_framework} and the associated quasi-interpolation operators.
$H^2$-conforming virtual elements are well established; see, e.g., \cite{brezzi2013,beirao2014,antonietti2021,antonietti2020,beirao2020,chen2022conforming,brenner2019,chen2022hessian}.
For simplicity, we restrict the description to the $d=2$ case. The higher-dimensional spaces can be constructed hierarchically in the space dimension with no additional difficulties~\cite{bonnet2025,chen2022conforming}. The construction relies on the VEM enhancement technique with general computable projection operators~\cite{ahmad2013,cangiani2017}.

\subsection{Virtual element spaces} \label{sec:projections}

First, on each  $K\in \cK_h$ a local enlarged virtual element space is introduced as
\begin{align*}
\widetilde{V}_K:=\left\{
v\in H^2(K)\, : \, \Delta^2 v\in \PP_m(K)
 \text{ and } v_{|s}\in \PP_{3\lor m}(s),\,
\partial_{\mathbf{n}}v_{|s}\in \PP_{m-1}(s),\,
\forall s\subset\partial K
\right\}.
\end{align*}
The following set $\dofs_K$ of local degrees of freedom is  incomplete for $\widetilde{V}_K$ but will characterize the  forthcoming (reduced) virtual element space:
\begin{align}
    \notag
    &h_{{\rm v}}^j\nabla^j v({\rm v}),\, j=0,1,\, \text{for all vertices } \rm v \text{ of } K; \\
    \notag
    &|s|^{j-1} \int_s \frac{\partial^j v}{\partial \mathbf{n}^j} q \, d\ell,\, \text{for all } q\in\MM_{m+j -4}(s),\, j=0,1,\, \text{for all edges } s \text{ of } K; \\
    \label{eq:cell_dofs}
    &|K|^{-1}\int_K v q\, dx,\, \text{for all } q\in\MM_{m-4}(K).
\end{align}
Here, for given $\omega\subset \RR^d$, 
the set $\MM_{k}(\omega)$ denotes a basis of $\PP_{k}(\omega)$ if $k\ge 0$ while
$\MM_{k}(\omega)=\emptyset$ if $k<0$.

The local virtual element space is obtained from $\widetilde{V}_K$ with the aid of a projection operator fully computable from the above degrees of freedom. Here we make a specific choice, referring to~\cite{ahmad2013,chen2022conforming,cangiani2017,dedner2022fourthorder} for the general idea.  Let $\Pi_{K,*}^m:\widetilde{V}_K\rightarrow \PP_m(K)$ an  operator that (a) reduces to the identity over the polynomial subspace $\PP_m(K)$, and (b) satisfies the constraints  $\int_K\Pi_{K,*}^mvq=\int_Kvq$ for $q\in\PP_{m-4}(K)$. Crucially,  both properties only depend on the degrees of freedom belonging to $\dofs_K$, and hence this operator is well-defined from these. Indeed,  the space of polynomials $\PP_m(K)$ is uniquely determined by the degrees of freedom and the constraints in (b) coincide with the elemental moments \eqref{eq:cell_dofs}.

A local virtual element space with the above as degrees of freedom is now defined as
\begin{align*}
V_K:=\left\{
v\in \widetilde{V}_K\, : \int_K (v- \Pi_{K,*}^mv)q\, dx=0, \forall q\in \PP_{m}(K)\setminus\PP_{m-4}(K)
\right\}.
\end{align*}
Note that, by definition of $V_K$ and the specific choice made for the  projection operator $\Pi_{K,*}^m$ it follows that the restriction of  $\Pi_{K,*}^m$ on ${V_K}$ coincides with the $L^2$-projection onto $\PP_{m}(K)$; henceforth we shall write $\Pi_K^m$ in place of $\Pi_{K,*}^m$.

Exploiting the computability  of $L^2$-projections of order up to $m$, we can also hierarchically compute $L^2$-projections of the gradient and hessian up to prescribed orders. 
Indeed, from $\Pi_K^{m_0}$, $m_0\le m$, for any given $m_1\le m_0+1$ we can compute the gradient projection into $(\PP_{m_1}(K))^d$ by solving
\begin{align*}
  \int_K \Pi_K^{m_1}\nabla v \cdot  q = -\int_K\Pi_K^{m_0} v {\rm div}q + \int_{\partial K} vq\cdot n,\qquad\text{for all}\; q\in(\PP_{m_1}(K))^d,
\end{align*}
and,  for any given $m_2\le m_1+1$, we can compute the hessian projection into $(\PP_{m_2}(K))^{d\times d}$ by solving
\begin{align*}
  \int_K \Pi_K^{m_2}\nabla^2 v\colon q
     = -\int_K\Pi_K^{m_1}\nabla v\cdot {\rm div}q
       + \int_{\partial K} \nabla v\cdot qn
       ,\qquad\text{for all}\; q\in(\PP_{m_2}(K))^{d\times d},
\end{align*}
 for any $v\in V_K$. This approach allows for the implementation of all the variants of the VEM discussed in the previous sections.
Computability of the right-hand sides in the above follows from the observation that $v_{|\partial K}$ is a piecewise polynomial uniquely defined by the degrees of freedom and so consequently tangential derivatives are also computable. Furthermore, the normal derivative at a given edge of $K$ is also a polynomial uniquely defined by the given degrees of freedom.
Note that for $d>2$ the boundary integrals would need to involve additional face projections.

The global virtual element space $V_h \subset H^2(\Omega)$ may now be defined as
\begin{equation*}
    V_h := \{v \in H^2(\Omega) \mid v_{|K} \in V_K,\, \forall K \in \cK_h\}.
\end{equation*}
Recall also that $ V_h^0 := \{v \in V_h \mid v_{|\partial \Omega} = 0\}$ and $V_h^g := g_I + V_h^0$, where $g_I$ is an admissible approximation of $g$ to be fixed later. 

The proof that the space $V_h$ satisfies the required inverse inequalities and stabilization form scaling given, respectively, as Assumption~\ref{assum:inverse_ineq} and Assumption~\ref{assum:stab_scaling}, can be found in \cite{chen2022conforming}.

Finally,  the VEM is realized by fixing symmetric bilinear forms $s_{0,h}$ and $s_{2,h}$. Here we choose the so called \emph{D-recipe}~\cite{beirao2017,mascotto2018}, a modification of the standard \emph{dofi-dofi} recipe~\cite{bonnet2025} more appropriate for moderately high orders,   obtained by weighting each stabilization form with the diagonal of the corresponding consistency form.

\subsection{Order of convergence}

We derive from Theorem~\ref{thm:error_estimate} optimal orders of convergence under minimal regularity assumptions for the VEM realization described above. We rely on the  optimal Scott-Zhang type quasi interpolant presented in Section 5 of~\cite{bonnet2025}.    It is indeed shown in~\cite{bonnet2025} that there exists an operator $(\cdot)_I:H^{2}(\Omega)\rightarrow V_h$ with optimal approximation properties as follows. Let $v\in H^2(\Omega)$. For any $K\in \cK_h$ and $1 \leq s \leq m$, if $v \in H^{s+1}(\omega_K)$, then
    \begin{equation} \label{eq:interp}
        \|v - v_I\|_{2,K} \lesssim h_K^{s-1} |v|_{s+1, \omega_K}.
    \end{equation}
Using this together with standard projection error estimates  we derive from Theorem~\ref{thm:error_estimate} the following.

\begin{theorem}[rate of convergence]
    \label{thm:convergence_rate}
    Assume that $\eta\ge\eta^*$. Let $u \in V$ denote the unique solution to \eqref{eq:var} and  $u_h \in V_h$ denote  the unique solution to any of the VEM in \eqref{eq:scheme} with $g_I$ as above.  If $u \in H^{s+1}(\Omega)$ for some $1 \leq s \leq m$, then
    \begin{equation*}
        \|u - u_h\|_2 \lesssim \eta \kappa^{-1} (1 + \lambda) h^{s-1} \|u\|_{s+1}.
    \end{equation*}
\end{theorem}

\begin{proof}
    Observe that
    \begin{equation*}
        \|u - u_h\|_2
        \lesssim \|\Delta (u - u_h)\|_0
        \lesssim |u - u_h|_2
        \leq |u - u_h|_{*, \lambda},
    \end{equation*}
    where the first inequality is the a priori estimate for the Poisson equation, using that $(u - u_h)_{|\partial \Omega} = 0$. Then, using Theorem~\ref{thm:error_estimate},
    \begin{equation*}
        \|u - u_h\|_2 \lesssim \eta \kappa^{-1} (1 + \lambda) (\|u - u_I\|_2 + \|u - \Pi_h^m u\|_2 + \|u - \Pi_h^{m-2} u\|_0 + \|\nabla u - \Pi_h^{m-2} \nabla u\|_0),
    \end{equation*}
    with $u_I \in V_h$ obtained using the quasi interpolant $(\cdot)_I$ discussed above. One easily concludes using \eqref{eq:interp} and standard polynomial projection estimates.
\end{proof}

\begin{remark}
    For conciseness, the dependence on $\lambda$ is not tracked as precisely in Theorem~\ref{thm:convergence_rate} as in Theorem~\ref{thm:error_estimate}.
\end{remark}

\begin{remark}
    Similar reasoning can be used to derive rates of convergence from Theorem~\ref{thm:error_estimate_alt1} and Theorem~\ref{thm:nitsche_error_estimate}. In the case of Theorem~\ref{thm:error_estimate_alt1}, the factor $\kappa^{-1}$ becomes $\kappa^{-2}$.
\end{remark}

\section{Numerical experiments}
\label{sec:numerics}

In this section we investigate numerically the approximation properties of the
presented VEM approximations (and a DG method) on a series of (criss) simplex grids and Voronoi grids.
The number of elements
on each level of the grid series is reported in Table~\ref{tab:gridsize}
together with the number of degrees of freedom for the VEM method (note
that all VEM methods of the same order have the same number of
degrees of freedom).

The numerical tests are based on the Dune-Vem\footnote{\url{pypi.org/project/dune-vem/}}
software package which is part of the Dune framework
\cite{bastian2021,dedner2010,dedner2020}.
All our VEM approximations can be implemented in a unified fashion once a suitable operator $\Pi_{K,*}^m$ has been chosen. In our implementation we follow the \emph{constrained least squares} approach first described in \cite{dedner2022fourthorder,dedner2024}. For the $H^2$-conforming VEM spaces with $m\geq 3$, $\Pi_{K,*}^m$ is the unique solution to the constrained least squares (CLS) problem
\begin{align*} 
  & \sum_{\delta\in\dofs_K} \big(\delta(\Pi_{K,*}^mv - v)\big)^2 \to \text{min} \\
  \text{satisfying}\; & \int_K (\Pi_{K,*}^mv - v)q = 0\; \quad \text{for all}\; q\in \PP_{m-4}(K) 
\end{align*}
where $\dofs_K$ is the set of local degrees of freedom given in Section~\ref{sec:projections}.
For $m=2$, we define $\Pi_{K,*}^2$ into $\PP_3(K)$ to match the order of $v_{|\partial K}$ and, to render the above CLS problem uniquely solvable also on triangles where we only have $|\dofs_K|=9<{\rm dim }\PP_3(K)$, we impose the additional constraint
\begin{align*}
  \int_K\Delta\Pi_{K,*}^2 v - \int_{\partial K}\nabla v\cdot n = 0.
\end{align*}
Details on this construction and on the solution of the resulting CLS can be found in \cite{dedner2024}.

We solve the nonlinear problem \eqref{eq:scheme} by means of Howard's algorithm, an iterative method where, given an approximation $u_h^n$, a new approximation is given by $u_h^{n+1}=u_h^n-\delta_h^n$ where $\delta_h^n\in V_h$ solves the linear problem
\begin{align} \label{eq:linHowards}
    \int_\Omega F^n_{\gamma, h}[\delta_h^{n+1}] \cL_{\lambda, h} w\, d x
         + \eta\<\cA_h^s(\delta_h^{n+1}), w\> \, = \,
     \<R_h^n,w\>
\end{align}
where the right-hand side is the residual
\begin{align*}
  \<R_h^n,w\>:=\int_\Omega F^n_{\gamma, h}[u_h^{n}] \cL_{\lambda, h} w\, d x + \eta\<\cA_h^s(u_h^{n}), w\>.
\end{align*}
At step $n$, the linear operator $F^n_{\gamma,h}$ is given by
\begin{align*}
  F^n_{\gamma, h}[v] := \gamma^{\alpha^n} (\cL_h^{\alpha^n} v - f^{\alpha^n}),
\end{align*}
where, for $x\in\Omega$, we define a function $\alpha^n\colon\Omega\to\Lambda_1$ such that
\begin{align} \label{eq:minHowards}
  \gamma^{\alpha^n(x)} \Big(\cL_h^{\alpha^n(x)}u_h^n(x) - f^{\alpha^n(x)}(x)\Big)
  &= \inf_{\alpha \in \Lambda_1} \Big[\gamma^{\alpha} \Big((\cL_h^{\alpha}u_h^n)(x) - f^{\alpha}(x)\Big)\Big].
\end{align}

{\bf Note:} in all our tests, $\Lambda_2=\emptyset$. Consequently, we only need to focus on minimizing $\alpha\in\Lambda_1$ in Howard's algorithm.

In practice, during assembly of the matrix and right side to compute $\delta_h^n$, the integral in \eqref{eq:linHowards} is replaced by a quadrature rule with points $(x_q)_q$. In each of these points the minimization problem \eqref{eq:minHowards} is solved either directly in the case that $\Lambda_1$ is finite or by means of an iterative method like golden section search or using a suitable solver from the \emph{NLopt nonlinear-optimization package}\footnote{\url{http://github.com/stevengj/nlopt}} by Steven G.~Johnson.
We initialize Howard's algorithm with $u_h^0\equiv 0$ and terminate the iterations once the $H^2$ dual norm of the residual $R_h^n$ drops below $10^{-9}$. The dual norm of the residual is computed as $\<R_h^n,z_h\>$, where $z_h\in V_h$ is the solution to
\begin{align*}
  (z_h,v_h)_2 = \<R_h^n,v_h\>.
\end{align*}

We first focus on the virtual element method of Section~\ref{sec:scheme} in the case $j=0$, denoted below as VEM(m,m,m).
We investigate this method in some detail for three problems on both the simplex and voronoi grid sequences, and use it as reference method when investigating the other approaches
discussed in this paper. Since the results are very similar across the test cases and grid sequences, we report only representative results, focusing primarily on the second problem on simplex grids; other combinations are omitted when they lead to analogous conclusions.
We conclude, in the final section, with a more detailed investigation of the third problem, which features a more dominant first order term.
All the computations with the virtual element method were carried out
with a stabilization constant of $\eta=1$.

Our error estimates focus on the error in $|\cdot|_{*,\lambda}$ which
(depending on the value of $\lambda$) is a combination of the broken
$H^2$, $H^1$, and $L^2$ (semi)norms. However, we report below results in these broken (semi)norms individually, noting that the $|\cdot|_{*,\lambda}$ (semi)norm
is always close to the $H^2$ seminorm. 
Since in the VEM we do not have direct access to the values of the numerical solution $u_h$, we compute errors involving local polynomial projections. Accordingly, by $H^2$, $H^1$, and $L^2$ errors, we mean respectively
\begin{align*}
   &\|\nabla^2 u - \Pi_h^{m}\nabla^2 u_h\|_0, &
   &\|\nabla u - \Pi_h^{m}\nabla u_h\|_0, &
   &\|u - \Pi_h^{m}u_h\|_0.
\end{align*}

Note that, in later experiment, we always use the maximum projection order $m$ in the above projections,
independently of the projection orders (e.g.\ $m-j$) used in the computation of $u_h$. This
ensures that the errors are computed in the same way for all methods.

\begin{table}
\centering
\begin{tabular}{rrrrrr}
\hline
\multicolumn{6}{c}{simplex grids} \\
\hline
level & elements & dofs $m=2$ & dofs $m=3$ & dofs $m=4$ & dofs $m=5$ \\
1 & 128    & 243    & 451   & 275   & 451   \\
2 & 512    & 867    & 1667  & 995   & 1667  \\
3 & 2048   & 3267   & 6403  & 3779  & 6403  \\
4 & 8192   & 12675  & 25091 & 14723 & 25091 \\
5 & 32768  & 49923  & 99331 & 58115 &       \\
\hline
\multicolumn{6}{c}{Voronoi grids} \\
\hline
level & elements & dofs $m=2$ & dofs $m=3$ & dofs $m=4$ & dofs $m=5$ \\
1 & 47    & 288    & 430    & 761     & 1139   \\
2 & 196   & 1182   & 1771   & 3145    & 4715   \\
3 & 784   & 4710   & 7063   & 12553   & 18827  \\
4 & 3135  & 18816  & 28222  & 50169   & 75251  \\
5 & 12544 & 75270  & 112903 & 200713  &        \\
\hline
\end{tabular}

\caption{Number of elements and degrees of freedom (dofs) for the tested
         VEM. The simulations were stopped after a certain number of
         refinement steps or when the number of dofs exceeded $10^5$.}
\label{tab:gridsize}
\end{table}


\subsection{Reference virtual element method}
As announced, we start by studying the method VEM(m,m,m).

\emph{Problem 1: First experiment from \cite{smears2014}.}
For this test case we take $\Lambda_1 = [0, \pi / 3] \times \SO(2)$. For
$\theta\in [0, \pi / 3]$ and a fixed rotation $R=R(\omega)\in \SO(2)$
of angle $\omega\in [0,2\pi)$, we define
\begin{equation*}
    A^{(\theta, R)}(x) = \frac{1}{2} R^\top A_0(\theta) R, \quad \text{where } A_0(\theta) := \begin{pmatrix}1 + \sin^2(\theta) & \sin \theta \cos \theta \\ \sin \theta \cos \theta & \cos^2 \theta\end{pmatrix},
\end{equation*}
$b^{(\theta, R)}(x) = 0$, $c^{(\theta, R)}(x) = \pi^2$, and $f^{(\theta, R)}(x) = \sqrt{3} \sin^2 \theta / \pi^2 + f_0(x)$ with $f_0$ chosen so that the exact solution is given by $u(x_1, x_2) = \exp(x_1 x_2) \sin(\pi x_1) \sin(\pi x_2)$ for $x = (x_1, x_2) \in \Omega = (0, 1)^2$.

In the design of the scheme, we choose $\gamma^{\alpha, \beta} = \gamma^\alpha = \gamma^{(\theta, R)}$ according to \eqref{eq:gamma}, and, following \cite[Example~1]{smears2014}, $\lambda = (8/7) \pi^2$ and $\kappa = 1 - \sqrt{1 - 1/7} \approx 0.074$.

As can be seen in Figure~\ref{fig:p3lllother} (left) and from the
experimental convergence rates summarized in the left most columns of the
tables in Table~\ref{tab:p3eoc}, our reference VEM method
converges with the order of $m-1$ in the $H^2$ error,
which is expected from our analysis.
This behavior is observed on both the simplex grid sequence (top) and the
Voronoi grids (bottom) for polynomial orders $m=2$, $3$, $4$, and $5$.
We also report, in the same figure and table, the $H^1$ and $L^2$ errors and
convergence rates,
for which we do not have a convergence analysis. These results
show convergence of order $m$ for the $H^1$ error on both grid
sequences. In the $L^2$ norm, we observe convergence close to $m+1$ for $m=3$, $4$, and $5$,
while for $m=2$ the convergence rate in $L^2$ is the same as in $H^1$.

\begin{figure}[ht!]
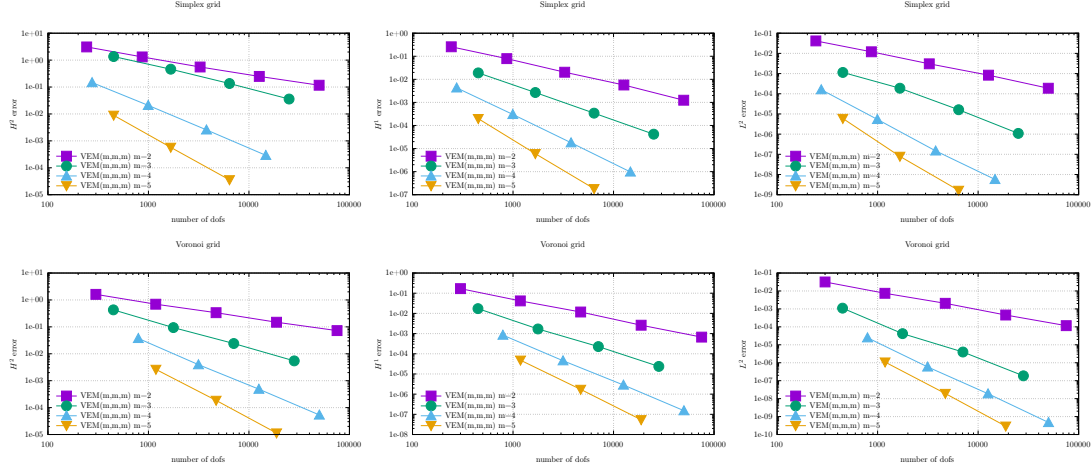

\centering
\resizebox{0.32\textwidth}{!}{\input{results/A3/eps/vem0A3simplexErrH2}}
  \resizebox{0.32\textwidth}{!}{\input{results/A3/eps/vem0A3simplexErrH1}}
  \resizebox{0.32\textwidth}{!}{\input{results/A3/eps/vem0A3simplexErrL2}}

\resizebox{0.32\textwidth}{!}{\input{results/A3/eps/vem0A3voronoiErrH2}}
  \resizebox{0.32\textwidth}{!}{\input{results/A3/eps/vem0A3voronoiErrH1}}
  \resizebox{0.32\textwidth}{!}{\input{results/A3/eps/vem0A3voronoiErrL2}}
\caption{Problem 1: $H^2$, $H^1$, $L^2$
         errors (from left to right) on the simplex grids (top) and
         Voronoi grids (bottom)
         for the reference method VEM(m,m,m).}
\label{fig:p3lllother}
\end{figure}
\begin{table}
\centering
\makebox{\centering \begin{tabular}{rrr}
\hline
   $H^2$ &   $H^1$ &   $L^2$ \\
\hline
    1.23 &    1.71 &    1.79 \\
    1.24 &    1.95 &    1.98 \\
    1.16 &    1.84 &    1.87 \\
    1.09 &    2.19 &    2.15 \\
\hline
\end{tabular} }\hfill%
\makebox{\centering \begin{tabular}{rrr}
\hline
   $H^2$ &   $H^1$ &   $L^2$ \\
\hline
    1.56 &    2.83 &    2.6  \\
    1.77 &    2.97 &    3.54 \\
    1.92 &    3.03 &    3.89 \\
\hline
\end{tabular} }\hfill%
\makebox{\centering \begin{tabular}{rrr}
\hline
   $H^2$ &   $H^1$ &   $L^2$ \\
\hline
    2.81 &    3.82 &    4.89 \\
    3.03 &    4.08 &    5.18 \\
    3.14 &    4.21 &    4.7  \\
\hline
\end{tabular} }\hfill%
\makebox{\centering \begin{tabular}{rrr}
\hline
   $H^2$ &   $H^1$ &   $L^2$ \\
\hline
    3.96 &    5.05 &    6.23 \\
    4.01 &    5.04 &    5.6  \\
\hline
\end{tabular} }\\[1em]
\makebox{\centering \begin{tabular}{rrr}
\hline
   $H^2$ &   $H^1$ &   $L^2$ \\
\hline
    1.29 &    2.13 &    2.22 \\
    1.12 &    1.95 &    1.98 \\
    1.13 &    2.09 &    2.09 \\
    1.14 &    2.16 &    2.17 \\
\hline
\end{tabular} }\hfill%
\makebox{\centering \begin{tabular}{rrr}
\hline
   $H^2$ &   $H^1$ &   $L^2$ \\
\hline
    2.32 &    3.54 &    4.99 \\
    2.1  &    3.1  &    3.63 \\
    2.06 &    3.14 &    4.24 \\
\hline
\end{tabular}
 }\hfill%
\makebox{\centering \begin{tabular}{rrr}
\hline
   $H^2$ &   $H^1$ &   $L^2$ \\
\hline
    3.43 &    4.46 &    5.71 \\
    3.23 &    4.27 &    5.3  \\
    3.06 &    4.07 &    5.12 \\
\hline
\end{tabular} }\hfill%
\makebox{\centering \begin{tabular}{rrr}
\hline
   $H^2$ &   $H^1$ &   $L^2$ \\
\hline
    4.05 &    5.08 &    6.15 \\
    4.3  &    5.34 &    6.46 \\
\hline
\end{tabular} }\iftoggle{author}{}{\\[1em]}
\caption{Problem 1: experimental orders of convergence on the simplex grids
(top row) and the Voronoi grids (bottom row) for $m=2$, $3$, $4$, $5$ (left to right). }
\label{tab:p3eoc}
\end{table}

\emph{Problem 2: checkerboard problem (\cite[Test~4]{wu2019thesis}).}
For this test case we take $\Lambda_1 = \{0,1\}$,
\begin{equation*}
    A^{\alpha} = \begin{pmatrix}2+\alpha & 1 \\ 1 & 1+\alpha\end{pmatrix},
\end{equation*}
$b^{\alpha}=0, c^{\alpha}=0$, and $f^{\alpha}(x) = f(x)$ chosen so that
the exact solution is given by
$u(x_1,x_2) = \sin(x_1)\sin(x_2)$ with
$x\in\Omega=(-\pi,\pi)^2$. This leads to
\begin{equation*}
    f(x) = \begin{cases}
          H(x) & x_1x_2>0, \\
          H(x) - 2\sin(x_1)\sin(x_2) & x_1x_2<0,
          \end{cases}
\end{equation*}
with $H(x_1,x_2) = -3\sin(x_1)\sin(x_2) + 2\cos(x_1)\cos(x_2)$
while the optimal control parameter is
\begin{equation*}
    \alpha(x) = \begin{cases}
          0 & x_1x_2>0, \\
          1 & x_1x_2<0.
          \end{cases}
\end{equation*}

We choose $\gamma^{\alpha, \beta} = \gamma^\alpha$ as in \eqref{eq:gamma}, and, following \cite{wu2019thesis}, $\lambda = 0$ and $\kappa = 1 - \sqrt{1 - 2/7} \approx 0.155$.

Results for Problem 2 are summarized in Figure~\ref{fig:p2lllother} and
Table~\ref{tab:p2eoc}. The errors show the expected order $m-1$ for
the $H^2$ error for both the simplex and Voronoi grid sequences.
As before, higher order convergence can be observed for the $H^1$ and $L^2$ errors.
The numerical convergence rates confirm our observations from Problem 1.
However, in this case we see some stability issues with the higher order virtual element methods
on the finest grids when the $L^2$ error starts dropping below $10^{-8}$,
which is close to the stopping criterion used in the Howard's algorithm.

\begin{figure}[ht!]
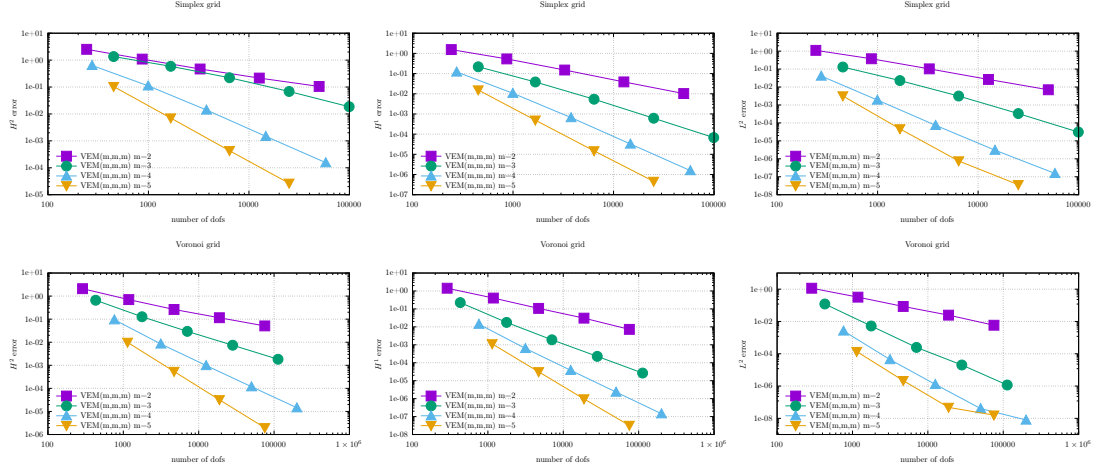

\centering
\resizebox{0.32\textwidth}{!}{\input{results/A2/eps/vem0A2simplexErrH2}}
  \resizebox{0.32\textwidth}{!}{\input{results/A2/eps/vem0A2simplexErrH1}}
  \resizebox{0.32\textwidth}{!}{\input{results/A2/eps/vem0A2simplexErrL2}}

\resizebox{0.32\textwidth}{!}{\input{results/A2/eps/vem0A2voronoiErrH2}}
  \resizebox{0.32\textwidth}{!}{\input{results/A2/eps/vem0A2voronoiErrH1}}
  \resizebox{0.32\textwidth}{!}{\input{results/A2/eps/vem0A2voronoiErrL2}}

\caption{Problem 2: $H^2$, $H^1$, and $L^2$
         errors (left to right) on the simplex grids (top) and
         Voronoi grids (bottom)
         for the reference method VEM(m,m,m).}
\label{fig:p2lllother}
\end{figure}
\begin{table}
\centering
\makebox{\centering \begin{tabular}{rrr}
\hline
   $H^2$ &   $H^1$ &   $L^2$ \\
\hline
    1.21 &    1.53 &    1.56 \\
    1.22 &    1.83 &    1.87 \\
    1.12 &    1.95 &    1.96 \\
    1.03 &    1.92 &    1.91 \\
\hline
\end{tabular} }\hfill%
\makebox{\centering \begin{tabular}{rrr}
\hline
   $H^2$ &   $H^1$ &   $L^2$ \\
\hline
    1.2  &    2.5  &    2.51 \\
    1.41 &    2.84 &    2.87 \\
    1.7  &    3.13 &    3.24 \\
    1.88 &    3.18 &    3.43 \\
\hline
\end{tabular} }\hfill%
\makebox{\centering \begin{tabular}{rrr}
\hline
   $H^2$ &   $H^1$ &   $L^2$ \\
\hline
    2.52 &    3.5  &    4.44 \\
    2.98 &    4.01 &    4.68 \\
    3.25 &    4.35 &    4.6  \\
    3.25 &    4.38 &    4.3  \\
\hline
\end{tabular} }\hfill%
\makebox{\centering \begin{tabular}{rrr}
\hline
   $H^2$ &   $H^1$ &   $L^2$ \\
\hline
    3.88 &    4.96 &    6.09 \\
    4.03 &    5.06 &    5.94 \\
    4.02 &    5.02 &    4.39 \\
\hline
\end{tabular} }\\[1em]
\makebox{\centering \begin{tabular}{rrr}
\hline
   $H^2$ &   $H^1$ &   $L^2$ \\
\hline
    1.43 &    1.66 &    1.67 \\
    1.46 &    1.97 &    1.97 \\
    1.31 &    1.98 &    1.96 \\
    1.12 &    1.97 &    1.97 \\
\hline
\end{tabular} }\hfill%
\makebox{\centering \begin{tabular}{rrr}
\hline
   $H^2$ &   $H^1$ &   $L^2$ \\
\hline
    2.16 &    3.33 &    4.11 \\
    2.17 &    3.31 &    4.54 \\
    2.19 &    3.37 &    3.97 \\
    1.92 &    2.93 &    3.9  \\
\hline
\end{tabular} }\hfill%
\makebox{\centering \begin{tabular}{rrr}
\hline
   $H^2$ &   $H^1$ &   $L^2$ \\
\hline
    3.17 &    4.08 &    5.35 \\
    3.2  &    4.17 &    5.26 \\
    3.36 &    4.42 &    5.46 \\
    2.87 &    3.8  &    2.23 \\
\hline
\end{tabular} }\hfill%
\makebox{\centering \begin{tabular}{rrr}
\hline
   $H^2$ &   $H^1$ &   $L^2$ \\
\hline
    3.81 &    4.7  &    5.44 \\
    4.2  &    5.17 &    5.67 \\
    4.42 &    5.38 &    1.75 \\
\hline
\end{tabular} }\iftoggle{author}{}{\\[1em]}
\caption{Problem 2: experimental orders of convergence on the simplex grids
(top row) and the Voronoi grids (bottom row) for $m=2$, $3$, $4$, $5$ (left to right). }
\label{tab:p2eoc}
\end{table}

\emph{Problem 3: advection problem.}
As a final test problem, we consider a variation of Problem 1 inspired by another
test case studied in \cite{smears2013}.
Compared to the setting of Problem 1, we consider different coefficients $b$ and $c$, now choosing $b=(0,\frac1\delta)$ and $c=(\frac\pi\delta)^2$ for some constant $\delta > 0$, independently of $(\theta, R) \in \Lambda_1$, and a different function $f_0$ chosen so that the exact solution is given by
\begin{equation*}
  u(x_1,x_2) = 5x_1(1-x_1)\left(x_2 - \frac{\exp(x_2/\delta)-1}{\exp(1/\delta)-1}\right).
\end{equation*}

Again, we choose $\gamma^{\alpha, \beta} = \gamma^\alpha = \gamma^{(\theta, R)}$ as in \eqref{eq:gamma}. Following \cite[Example~1]{smears2014}, we fix $\lambda = \frac{8}{7} c = \frac{8}{7} (\frac{\pi}{\delta})^2$. One may compute that the optimal $\kappa$ for such $\lambda$ is $\kappa \approx 0.044$, independently of~$\delta$.

Results for Problem 3 with $\delta=0.1$ are shown in Figure~\ref{fig:p5lllother} and
Table~\ref{tab:p5eoc}. As before we observe convergence rates close to $m-1$
for the $H^2$ error. The rates for the $H^1$ and
$L^2$ errors is not quite so clear for this problem, not always reaching the orders
$m$ and $m+1$ in $H^1$ and $L^2$, respectively, obtained in the previous examples.
We will investigate a much smaller value for $\delta$, i.e., a more
strongly pronounced boundary layer in the final part of our investigations.

\begin{figure}[ht!]
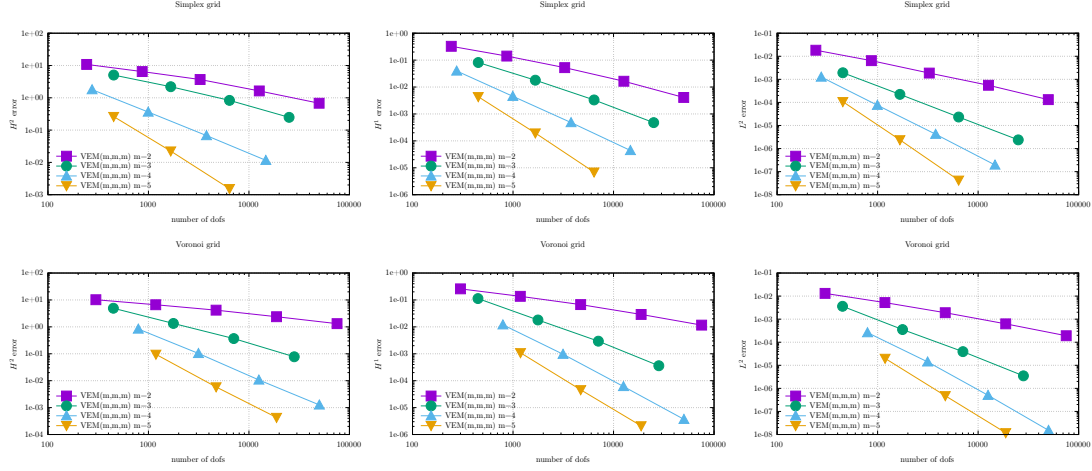

\centering
\resizebox{0.32\textwidth}{!}{\input{results/A5b/eps/vem0A5-0.1simplexErrH2}}
  \resizebox{0.32\textwidth}{!}{\input{results/A5b/eps/vem0A5-0.1simplexErrH1}}
  \resizebox{0.32\textwidth}{!}{\input{results/A5b/eps/vem0A5-0.1simplexErrL2}}

\resizebox{0.32\textwidth}{!}{\input{results/A5b/eps/vem0A5-0.1voronoiErrH2}}
  \resizebox{0.32\textwidth}{!}{\input{results/A5b/eps/vem0A5-0.1voronoiErrH1}}
  \resizebox{0.32\textwidth}{!}{\input{results/A5b/eps/vem0A5-0.1voronoiErrL2}}

\caption{Problem 3 with $\delta=0.1$: $H^2$, $H^1$, and $L^2$
         errors (left to right) on the simplex grids (top) and
         Voronoi grids (bottom)
         for the reference method VEM(m,m,m).}
\label{fig:p5lllother}
\end{figure}
\begin{table}
\centering
\makebox{\centering \begin{tabular}{rrr}
\hline
   $H^2$ &   $H^1$ &   $L^2$ \\
\hline
    0.74 &    1.21 &    1.5  \\
    0.82 &    1.43 &    1.79 \\
    1.17 &    1.68 &    1.75 \\
    1.27 &    2    &    2.07 \\
\hline
\end{tabular} }\hfill%
\makebox{\centering \begin{tabular}{rrr}
\hline
   $H^2$ &   $H^1$ &   $L^2$ \\
\hline
    1.2  &    2.18 &    3.11 \\
    1.41 &    2.44 &    3.26 \\
    1.75 &    2.8  &    3.29 \\
\hline
\end{tabular} }\hfill%
\makebox{\centering \begin{tabular}{rrr}
\hline
   $H^2$ &   $H^1$ &   $L^2$ \\
\hline
    2.3  &    3.12 &    4.06 \\
    2.42 &    3.24 &    4.17 \\
    2.54 &    3.45 &    4.43 \\
\hline
\end{tabular} }\hfill%
\makebox{\centering \begin{tabular}{rrr}
\hline
   $H^2$ &   $H^1$ &   $L^2$ \\
\hline
    3.55 &    4.48 &    5.5  \\
    3.87 &    4.82 &    5.81 \\
\hline
\end{tabular} }\\[1em]
\makebox{\centering \begin{tabular}{rrr}
\hline
   $H^2$ &   $H^1$ &   $L^2$ \\
\hline
    0.66 &    0.98 &    1.4  \\
    0.67 &    1.02 &    1.46 \\
    0.84 &    1.25 &    1.65 \\
    0.93 &    1.44 &    1.87 \\
\hline
\end{tabular} }\hfill%
\makebox{\centering \begin{tabular}{rrr}
\hline
   $H^2$ &   $H^1$ &   $L^2$ \\
\hline
    1.98 &    2.8  &    3.54 \\
    1.87 &    2.62 &    3.18 \\
    2.32 &    3.13 &    3.59 \\
\hline
\end{tabular} }\hfill%
\makebox{\centering \begin{tabular}{rrr}
\hline
   $H^2$ &   $H^1$ &   $L^2$ \\
\hline
    3.19 &    3.84 &    4.45 \\
    3.27 &    3.98 &    4.81 \\
    3.18 &    4.18 &    5.2  \\
\hline
\end{tabular} }\hfill%
\makebox{\centering \begin{tabular}{rrr}
\hline
   $H^2$ &   $H^1$ &   $L^2$ \\
\hline
    4.26 &    4.83 &    5.69 \\
    3.77 &    4.45 &    5.33 \\
\hline
\end{tabular} }\iftoggle{author}{}{\\[1em]}
\caption{Problem 3: experimental orders of convergence on the simplex grids
(top row) and the Voronoi grids (bottom row) for $m=2$, $3$, $4$, $5$ (left to right). }
\label{tab:p5eoc}
\end{table}

In the following sections, we investigate the other variants of the method
discussed in this paper. Since results for all problems are very similar,
we focus only on Problem 2. Also, as only a marginal difference in
performance was observed for all methods between the simplex and the Voronoi grids, we only report results relative to  the former grid sequence.
To make comparison between the methods
simpler, we treat the method VEM(m,m,m) discussed so far
as a reference method and include its errors in each plot.

\subsection{Comparison with lower order virtual element projections}

In this section, we compare the VEM(m,m,m) with the other choices of polynomial projection degrees introduced in section~\ref{sec:scheme} and~\ref{subsec:scheme_alt1}.

We denote by VEM(m-1,m-1,m-1) (respectively, VEM(m-2,m-2,m-2)) the method from Section~\ref{sec:scheme} with $j=1$ (respectively, $j=2$), and VEM(m,m-1,m-2) the one from Section~\ref{subsec:scheme_alt1}.

The results obtained with the simplex grid sequence are summarised in Figure~\ref{fig:p2vemprojectionsother}.
Overall little difference in the $H^2$ errors can be observed expect when using
the method VEM(m,m-1,m-2), and even in this case a significant
difference only occurs with $m=2$ (left row of Figure~\ref{fig:p2vemprojectionsother}).
The difference in error here is possibly due
to the higher stabilization constant used for this method due to the
additional factor $\kappa^{-1}$ as prediced by the analysis. As before, we also include the $H^1$ and $L^2$ errors.
Again, errors are similar for all methods except VEM(m,m-1,m-2) with $m=2$. In addition, some difference can be seen
in the $L^2$ norm for $m=3$ where the other equal  projection methods seem to outperform
the reference method. Convergence rates appear slightly higher on
coarser grids in this case. The reason for this will require further
investigation.

\begin{figure}[ht!]
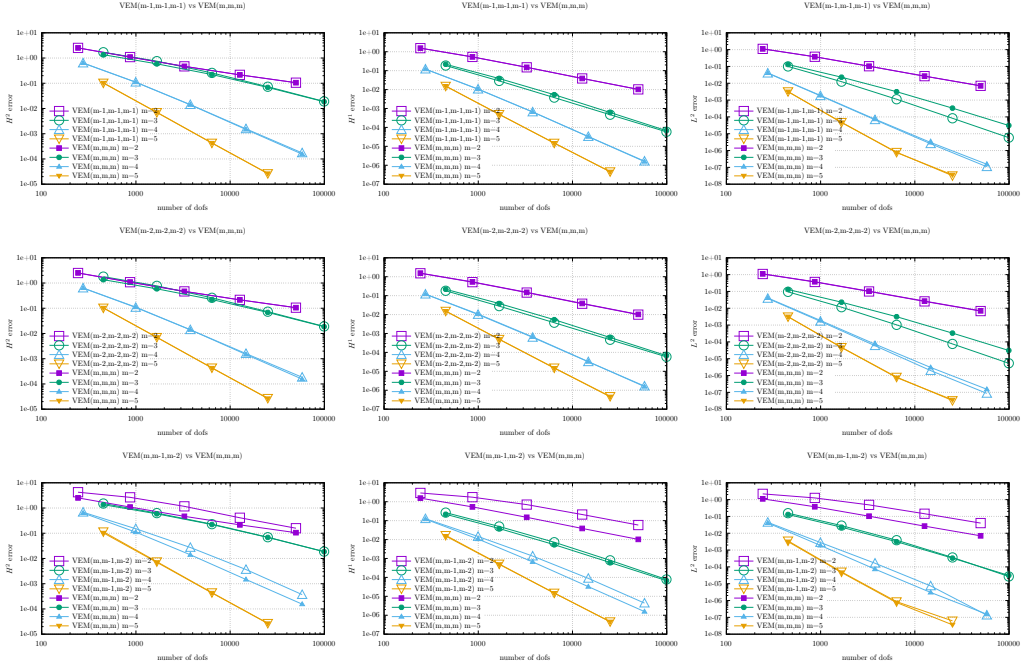

\centering
\resizebox{0.3\textwidth}{!}{\input{results/A2/eps/H2A2simplexEOC_vem1_vem0}}
  \resizebox{0.3\textwidth}{!}{\input{results/A2/eps/H1A2simplexEOC_vem1_vem0}}
  \resizebox{0.3\textwidth}{!}{\input{results/A2/eps/L2A2simplexEOC_vem1_vem0}}

\resizebox{0.3\textwidth}{!}{\input{results/A2/eps/H2A2simplexEOC_vem2_vem0}}
  \resizebox{0.3\textwidth}{!}{\input{results/A2/eps/H1A2simplexEOC_vem2_vem0}}
  \resizebox{0.3\textwidth}{!}{\input{results/A2/eps/L2A2simplexEOC_vem2_vem0}}

\resizebox{0.3\textwidth}{!}{\input{results/A2/eps/H2A2simplexEOC_vem-1_vem0}}
  \resizebox{0.3\textwidth}{!}{\input{results/A2/eps/H1A2simplexEOC_vem-1_vem0}}
  \resizebox{0.3\textwidth}{!}{\input{results/A2/eps/L2A2simplexEOC_vem-1_vem0}}
\caption{Problem 2: comparison between the reference method and virtual element approaches using lower order projections
         on simplex grids. From left to right:
         $H^2$ error, $H^1$ error, $L^2$ error.
         }
\label{fig:p2vemprojectionsother}
\end{figure}

\subsection{VEM with weakly imposed boundary conditions}
We repeat the same experiments as in the previous section comparing, this time, the reference method VEM(m,m,m) (with strongly imposed boundary conditions) and the same method but with weakly imposed boundary conditions (i.e.\ the method from Section~\ref{subsec:nitsche} with $j=0$), which we name VEM-Weak(m,m,m). 

Results for the errors in the $H^2$, $H^1$, and $L^2$ (semi)norms are shown in
Figure~\ref{fig:p2vemprojectionsweak}.
Errors are indistinguishable between the method with
weakly and strongly enforced boundary conditions, thus showing that weak imposition provides an effective and computationally simpler alternative. This holds for both the
simplex and the Voronoi grids and for all norms tested.

\begin{figure}[ht!]
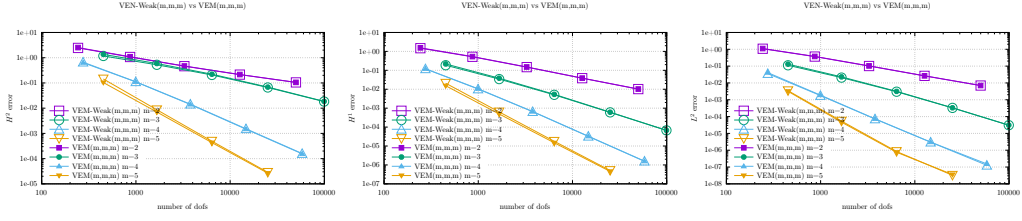

\centering
  \resizebox{0.3\textwidth}{!}{\input{results/A2/eps/H2A2simplexEOC_vembnd0_vem0}}
  \resizebox{0.3\textwidth}{!}{\input{results/A2/eps/H1A2simplexEOC_vembnd0_vem0}}
  \resizebox{0.3\textwidth}{!}{\input{results/A2/eps/L2A2simplexEOC_vembnd0_vem0}}

\caption{Problem 2: comparison of reference method on the simplex grids with VEM using weakly imposed
         boundary conditions.
         From left to right: $H^2$ error, $H^1$ error, $L^2$ error.}
\label{fig:p2vemprojectionsweak}.
\end{figure}

\subsection{Comparison with a Discontinuous Galerkin method}
We conclude our numerical investigation of the convergence of our methods comparing the reference method VEM(m,m,m)
with the DG method studied in \cite{smears2014}.
Numerical results for Problem 2 discretized by simplex grids are reported in  Figure~\ref{fig:p2vemprojectionsdgsimplex}.
The two methods are expected to deliver the same convergence rates, and this fact is confirmed by our results.
Comparing the two methods with the same mesh level and polynomial order, we observe very similar performance, especially for $m=4$ and $m=5$, with the DG method having a slight edge in some cases.
However, in terms of number of degrees of freedom, the VEM method is, with the exception of the $H^2$ error in the case $m=3$, always more efficient, with the biggest difference for $m=2$.
The advantage of the VEM methods is due to the lower number of
degrees of freedom they employ.
The DG method requires $\frac12(m+1)(m+2)$
degrees of freedom per element; the number of degrees of freedom for the
VEM method and the grid sizes are summarized in Table~\ref{tab:gridsize}.
For example, on the level 4 simplex grid the VEM spaces with orders
$m=2$, $3$, $4$, $5$ require $3267$, $6403$, $14723$, $25091$ degrees of freedom, respectively,
while the size of the corresponding DG spaces are $12288$, $20480$, $30720$, $43008$.


\begin{figure}[ht!]
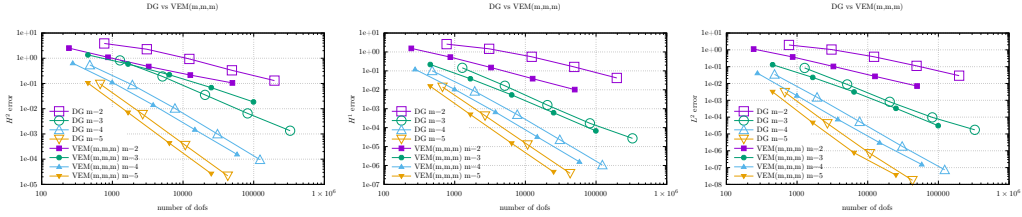

\centering
  \resizebox{0.3\textwidth}{!}{\input{results/A2/eps/H2A2simplexEOC_dg_vem0}}
  \resizebox{0.3\textwidth}{!}{\input{results/A2/eps/H1A2simplexEOC_dg_vem0}}
  \resizebox{0.3\textwidth}{!}{\input{results/A2/eps/L2A2simplexEOC_dg_vem0}}

\caption{Problem 2: comparison between the reference method and a DG method, on simplex grids.
         From left to right: $H^2$ error, $H^1$ error, $L^2$ error.}
\label{fig:p2vemprojectionsdgsimplex}.
\end{figure}

\subsection{Convergence of Howard's algorithm}
We finally investigate the convergence properties of Howard's
algorithm for the different methods, grid resolutions, and polynomial
orders. The results obtained with simplex grids are shown in
Figure~\ref{fig:p2iterations}.
We concentrate on the methods VEM(m,m,m) and VEM-Weak(m,m,m) (the two left columns) and also show, for comparison, results with the DG method (rightmost column).
The three methods show very similar convergence properties overall. In the lowest-order case $m=2$, they  require between 2 and 6 iterations depending on grid level, with a slight advantage for the DG method in a couple of cases. For $m>2$, performance is still similar across methods, with convergence reached in 3--6  steps except on the finest grids where 7 or 8 iterations are sometimes
required. The convergence is overall superlinear with the exception of
the results on the finest grids when $m=4$ or $m=5$. This could be due to
stability issues caused by the computation of the projection operators in
these cases.

\begin{figure}[ht!]
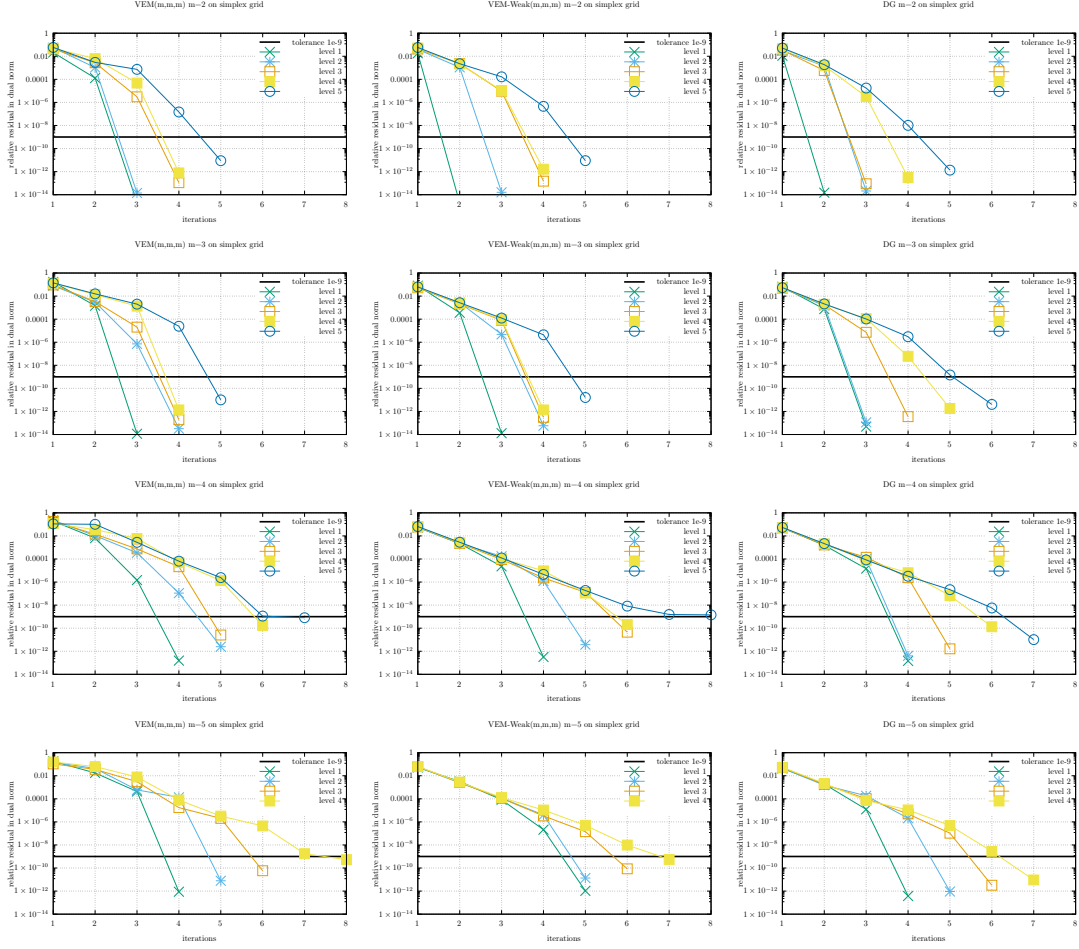

\centering
\resizebox{0.32\textwidth}{!}{\input{results/A2/eps/A2IterVEM0simplex2.tex}}
    \resizebox{0.32\textwidth}{!}{\input{results/A2/eps/A2IterVEMbnd0simplex2.tex}}
    \resizebox{0.32\textwidth}{!}{\input{results/A2/eps/A2IterDGsimplex2.tex}}

\resizebox{0.32\textwidth}{!}{\input{results/A2/eps/A2IterVEM0simplex3.tex}}
    \resizebox{0.32\textwidth}{!}{\input{results/A2/eps/A2IterVEMbnd0simplex3.tex}}
    \resizebox{0.32\textwidth}{!}{\input{results/A2/eps/A2IterDGsimplex3.tex}}

\resizebox{0.32\textwidth}{!}{\input{results/A2/eps/A2IterVEM0simplex4.tex}}
    \resizebox{0.32\textwidth}{!}{\input{results/A2/eps/A2IterVEMbnd0simplex4.tex}}
    \resizebox{0.32\textwidth}{!}{\input{results/A2/eps/A2IterDGsimplex4.tex}}

\resizebox{0.32\textwidth}{!}{\input{results/A2/eps/A2IterVEM0simplex5.tex}}
    \resizebox{0.32\textwidth}{!}{\input{results/A2/eps/A2IterVEMbnd0simplex5.tex}}
    \resizebox{0.32\textwidth}{!}{\input{results/A2/eps/A2IterDGsimplex5.tex}}

\caption{Problem 2: number of Howard's algorithm iterations for degrees
$m=2$, $3$, $4$, $5$ (top to bottom). Results are shown using the simplex grid
sequence with the reference method VEM(m,m,m), its counterpart with weak imposition of the boundary condition VEM-Weak(m,m,m), and the DG method (left to right).}
\label{fig:p2iterations}
\end{figure}

\subsection{Advection problem 3}
Finally, we compare the different methods in the setting of Problem 3 for different values of $\delta$. We test the virtual element method with $m=4$ and the same sequence of simplex grids as in the previous tests. Note,
in particular, that  we do not refine towards the top boundary to resolve the layer. We note that Howard's
algorithm converged in less than 5 iterations in all cases reported below with final residual errors of around $10^{-11}$.

As in the previous example, we investigate the $H^2$, $H^1$, and $L^2$ errors obtained by the different virtual element methods and the DG method with respect to number of
degrees of freedom. Results are summarized in Figure~\ref{fig:p3boundarylayer}.
Those obtained with the virtual element method with weak boundary
conditions are omitted since they are indistinguishable from the
ones relying on the strong imposition.
The reference method VEM(m,m,m) is always the most efficient. Compared to the DG method,
it achieves similar accuracy on any given grid in spite of using a considerably smaller number of
degrees of freedom. In contrast to the previous examples, using lower order projections leads to a significant increase in the error.
This is especially true for the method VEM(m-2,m-2,m-2). 
The method matching more closely the accuracy of the reference method is
VEM(m,m-1,m-2). However, in the case  $\delta=0.1$, the convergence rate of VEM(m,m-1,m-2) seems to degrade on finer grids.
The convergence of all methods on coarser grids is lower when  $\delta=0.0125$ than when
$\delta=0.1$. This issue is expected and due to the under resolved boundary layer; on
the finer grids convergence is again similar to what we observed for larger values of $\delta$.
\begin{figure}[ht!]
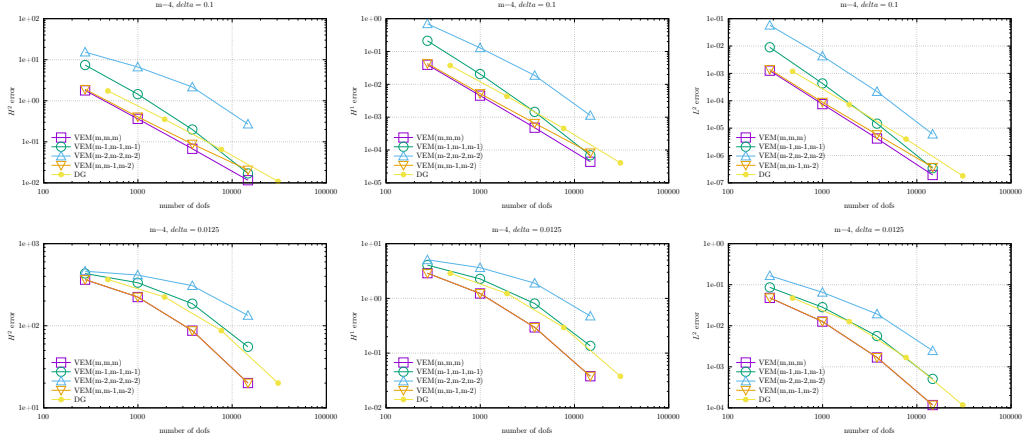

\centering
  \resizebox{0.3\textwidth}{!}{\input{results/A5b/eps/H2A5-1simplexEOC}}
  \resizebox{0.3\textwidth}{!}{\input{results/A5b/eps/H1A5-1simplexEOC}}
  \resizebox{0.3\textwidth}{!}{\input{results/A5b/eps/L2A5-1simplexEOC}}

  \resizebox{0.3\textwidth}{!}{\input{results/A5b/eps/H2A5-0125simplexEOC}}
  \resizebox{0.3\textwidth}{!}{\input{results/A5b/eps/H1A5-0125simplexEOC}}
  \resizebox{0.3\textwidth}{!}{\input{results/A5b/eps/L2A5-0125simplexEOC}}

\caption{Problem 3: advection problem computed on the given simplex grid with $m=4$ using from (top to bottom)
         $\delta=0.1,0.0125$. From left to right:
         $H^2$ error, $H^1$ error, $L^2$ error.}
\label{fig:p3boundarylayer}.
\end{figure}

We conclude our investigations by plotting the
resulting values of the optimization parameter $\alpha$ as
computed on level 2 and 4 of the simplex grid.
Figure~\ref{fig:p3alpha0.1} shows $\alpha$ in the case $\delta=0.1$.
\foreach \x in {0.1}{
\begin{figure}[ht!]
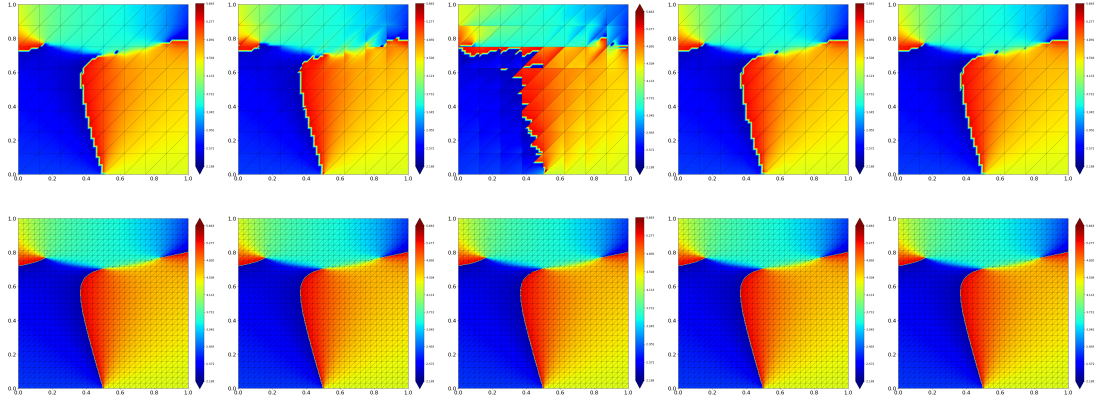

  \centering
  \includegraphics[width=0.19\textwidth]{results/A5b/kappaVEM_simplex_5A5-\x _8_4_0Alpha.png}
  \includegraphics[width=0.19\textwidth]{results/A5b/kappaVEM_simplex_5A5-\x _8_4_1Alpha.png}
  \includegraphics[width=0.19\textwidth]{results/A5b/kappaVEM_simplex_5A5-\x _8_4_2Alpha.png}
  \includegraphics[width=0.19\textwidth]{results/A5b/kappaVEM_simplex_5A5-\x _8_4_-1Alpha.png}
  \includegraphics[width=0.19\textwidth]{results/A5b/kappaDG_simplex_5A5-\x _8_4_0Alpha.png}

  \includegraphics[width=0.19\textwidth]{results/A5b/kappaVEM_simplex_5A5-\x _32_4_0Alpha.png}
  \includegraphics[width=0.19\textwidth]{results/A5b/kappaVEM_simplex_5A5-\x _32_4_1Alpha.png}
  \includegraphics[width=0.19\textwidth]{results/A5b/kappaVEM_simplex_5A5-\x _32_4_2Alpha.png}
  \includegraphics[width=0.19\textwidth]{results/A5b/kappaVEM_simplex_5A5-\x _32_4_-1Alpha.png}
  \includegraphics[width=0.19\textwidth]{results/A5b/kappaDG_simplex_5A5-\x _32_4_0Alpha.png}

\caption{Problem 3: values of $|(\theta,\omega)|$ on the level 2 (top row) and the
         finest level (bottom row) of the simplex grid sequence $m=4$ using $\delta=\x$.
         From left to right: VEM(m,m,m), VEM(m-1,m-1,m-1),
         VEM(m-2,m-2,m-2), VEM(m,m-1,m-2), and DG.
         Recall that $\alpha=(\theta,R(\omega))$ with $R(\omega)$
         describing rotation with angle $\omega$.}
\label{fig:p3alpha\x}
\end{figure}
}
While results look very similar on the finer grid (bottom row), differences between methods
are clearly visible on the coarser grid underlining the difference in
errors observed above. Figure~\ref{fig:p3alpha0.0125} shows results corresponding to
$\delta=0.0125$. Here we clearly see differences between the methods  even on the finer grids
(bottom row). Especially results with VEM(m-2,m-2,m-2) (middle plots) show considerably
more oscillations compared to all other methods.

\foreach \x in {0.0125}{
\begin{figure}[ht!]
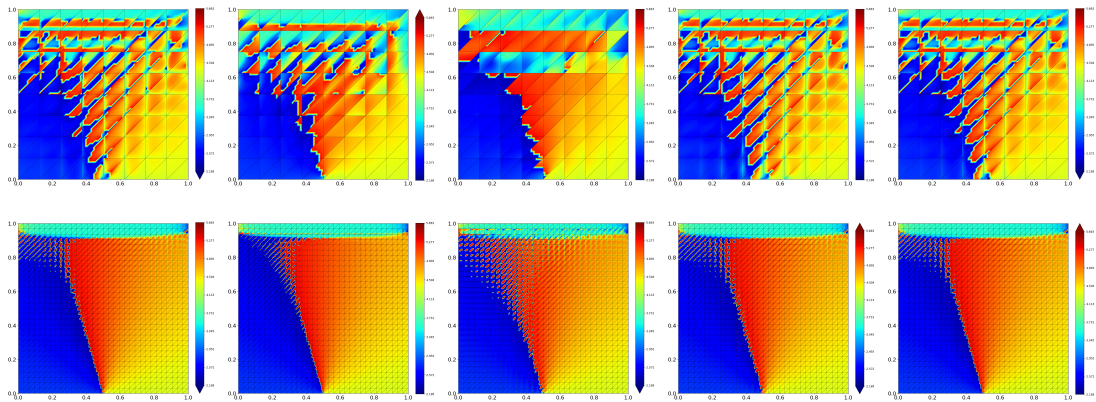

  \centering
  \includegraphics[width=0.19\textwidth]{results/A5b/kappaVEM_simplex_5A5-\x _8_4_0Alpha.png}
  \includegraphics[width=0.19\textwidth]{results/A5b/kappaVEM_simplex_5A5-\x _8_4_1Alpha.png}
  \includegraphics[width=0.19\textwidth]{results/A5b/kappaVEM_simplex_5A5-\x _8_4_2Alpha.png}
  \includegraphics[width=0.19\textwidth]{results/A5b/kappaVEM_simplex_5A5-\x _8_4_-1Alpha.png}
  \includegraphics[width=0.19\textwidth]{results/A5b/kappaDG_simplex_5A5-\x _8_4_0Alpha.png}

  \includegraphics[width=0.19\textwidth]{results/A5b/kappaVEM_simplex_5A5-\x _32_4_0Alpha.png}
  \includegraphics[width=0.19\textwidth]{results/A5b/kappaVEM_simplex_5A5-\x _32_4_1Alpha.png}
  \includegraphics[width=0.19\textwidth]{results/A5b/kappaVEM_simplex_5A5-\x _32_4_2Alpha.png}
  \includegraphics[width=0.19\textwidth]{results/A5b/kappaVEM_simplex_5A5-\x _32_4_-1Alpha.png}
  \includegraphics[width=0.19\textwidth]{results/A5b/kappaDG_simplex_5A5-\x _32_4_0Alpha.png}

\caption{Problem 3: values of $|(\theta,\omega)|$ on the level 2 (top row) and the
         finest level (bottom row) of the simplex grid sequence $m=4$ using $\delta=\x$.
         From left to right: VEM(m,m,m), VEM(m-1,m-1,m-1),
         VEM(m-2,m-2,m-2), VEM(m,m-1,m-2), and DG.
         Recall that $\alpha=(\theta,R(\omega))$ with $R(\omega)$
         describing rotation with angle $\omega$.}
\label{fig:p3alpha\x}
\end{figure}
}
\FloatBarrier

\section*{Acknowledgments}

We are indebted to I.~Smears and E.~Süli for discussions regarding the effect of quadrature in the convergence of policy iterations~\cite{hall2026}.

We also acknowledge partial support from the following funding agencies.
GB: NSF grant DMS-1908267.
AC:  European Union - Horizon Europe  grant  EuroHPC JU - 101172493. AC also acknowledges membership of INdAM Research group GNCS.
RHN: NSF grants DMS-1908267 and DMS-2512392.

\bibliographystyle{plain}
\bibliography{references}
\end{document}